\documentclass[reqno]{amsart}
\usepackage{amssymb}
\usepackage{graphicx}

\usepackage[usenames, dvipsnames]{color}
\usepackage{verbatim}
\usepackage{mathrsfs}

\numberwithin{equation}{section}

\newtheorem{theorem}{Theorem}[section]

\newtheorem{lemma}[theorem]{Lemma}
\newtheorem{prop}[theorem]{Proposition}

\theoremstyle{definition}
\newtheorem{remark}[theorem]{Remark}

\theoremstyle{definition}

\theoremstyle{definition}

\makeatletter
\def\dashint{\operatorname%
{\,\,\text{\bf-}\kern-.98em\DOTSI\intop\ilimits@\!\!}}
\makeatother

\def\\det{\text{det}}

\def\.5{\frac{1}{2}}

\newcommand{\RN}[1]{%
  \textup{\uppercase\expandafter{\romannumeral#1}}%
}

\renewcommand{\epsilon}{\varepsilon}

\newcounter{marnote}

\begin{document}
\title[Liouville theorem and existence]{Liouville theorem and sharp solvability for solutions of the parabolic Monge-Ampère equation with periodic data}
\author[K. Yan]{Kui Yan}
\address[K. Yan]{School of Mathematical Sciences, Beijing Normal University, Laboratory of Mathematics and Complex Systems, Ministry of Education, Beijing 100875, China.}
\email{202231130028@mail.bnu.edu.cn}

\author[J. Bao]{Jiguang Bao\textsuperscript{*}}
\address[J. Bao]{School of Mathematical Sciences, Beijing Normal University, Laboratory of Mathematics and Complex Systems, Ministry of Education, Beijing 100875, China.}
\email{jgbao@bnu.edu.cn}

\footnotetext{*corresponding author}

\date{\today} 

\begin{abstract}

We prove a Liouville Theorem for ancient solutions of the parabolic Monge-Ampère equation with smooth periodic data, generalizing Caffarelli-Li's result \cite{cl04} in 2004 to the parabolic background. To achieve this, we obtain a necessary and sufficient condition for the existence of the smooth periodic solution of the equation $\left(1-u_t\right)\det \left(D_x^2u+I\right)=f$ in $\mathbb{R}^{n+1}$, where $f$ is smooth and periodic in both spatial and temporal variables. This parabolic existence theorem parallels the elliptic counterpart established by Li \cite{l90} in 1990.

\quad

\noindent{\textbf{Keywords: }}parabolic Monge-Ampère equation, periodic solution, Liouville Theorem, periodic data.

\quad

\noindent{\textbf{Mathematics Subject Classification: }}35K96.
\end{abstract}
\maketitle

\section{introduction}\label{introduction}

This paper concerns a Liouville theorem and the existence of the periodic solution to a Monge-Ampère type fully nonlinear parabolic equation originated from
\begin{equation}\label{pma1}
-u_t\det D_x^2u=f,
\end{equation}
when $f=f(x,t)$ is periodic. Here we use $u_t$ to denote $u$'s derivative in $t$ and $D_x^2u$ to denote the Hessian of $u$ in $x$. A function $u=u(x,t)$ is termed p-convex if it is convex in $x$ and nonincreasing in $t$. Denote $e_1,\cdots,e_n$ as the orthonormal basis vectors in $\mathbb{R}^n$.

The parabolic Monge-Ampère equation \eqref{pma1} was introduced by Krylov in \cite{k76}. This equation bears close connections to ABP-type maximal principle for parabolic equations \cite{ks81}, stochastic process \cite{k81a,k81b} and deformation of the surface \cite{f74,t85}.

We state our Liouville theorem as follows.

\begin{theorem}\label{liouville}
Let $f\in C^{2+\alpha,\frac{2+\alpha}{2}}\left(\mathbb{R}^{n+1}_-\right)$ satisfy that
\[
f(x+ke_i,t-l)=f(x,t)
\]
for every $(x,t)\in\mathbb{R}^{n+1}_-$, $k,l\in\mathbb{N}$ and $1\le i\le n$, and $u\in C^{2,1}\left(\mathbb{R}^{n+1}_-\right)$ be a p-convex ancient solution to \eqref{pma1} with  
\begin{equation}\label{ut}
m_1\le-u_t\le m_2\quad\text{in}\quad\mathbb{R}^{n+1}_-,
\end{equation}
where $m_1,m_2>0$ are two constants. Then there exist a constant $\tau>0$, a vector $b\in\mathbb{R}^n$ and an $n\times n$ positive definite matrix $A$ with 
\[
\tau\det A=-\mkern-19mu\int_{[0,1]^n\times[-1,0]}f(x,t)dxdt,
\]
such that 
\[
v(x,t):=u(x,t)-\left(-\tau t+\frac{1}{2}x'Ax+b\cdot x\right)
\] 
possesses the same periodicity of $f$.
\end{theorem}

\begin{remark}
Theorem \ref{liouville} remains valid when the orthonormal basis vectors $e_1,\cdots,e_n$ are replaced by arbitrary linearly independent vectors $\epsilon_1,\cdots,\epsilon_n$ and the time period is changed into $\epsilon_0>0$. Specifically, if $f$ satisfies that 
\[
f(x+k\epsilon_i,t-l\epsilon_0)=f(x,t)
\]
for every $(x,t)\in\mathbb{R}^{n+1}_-$, $k,l\in\mathbb{N}$ and $1\le i\le n$, then there exist a constant $\tau>0$, a vector $b\in\mathbb{R}^n$ and an $n\times n$ positive definite matrix $A$ with 
\begin{equation}\label{coefficient}
\tau\det A=-\mkern-19mu\int_{\mathcal{C}_p}f(x,t)dxdt 
\end{equation}
such that $v$ inherits $f$'s periodicity, where
\[
\mathcal{C}_p:=\left\{\left(\sum\limits_{i=1}^n\lambda_i\epsilon_i,-\lambda_0\epsilon_0\right)\in\mathbb{R}^{n+1}_-:0\le\lambda_j\le1,0\le j\le n\right\}.
\]
\end{remark}

\begin{remark}
The Liouville theorem ($f\equiv1$) of Gutiérrez-Huang \cite{gh98} can be directly derived as special cases of Theorem \ref{liouville}. In this reaserch direction, Zhang-Bao \cite{zb18} proved that when $f=f(x)$ is a positive, periodic and Hölder continuous function in $\mathbb{R}^{n+1}_-$, any p-convex ancient solution must be of the form $-\tau t+p(x)+v(x)$, where $\tau>0$, $p(x)$ is a convex quadratic polynomial, and $v(x)$ is periodic. 
\end{remark}

When $f\in C^\alpha\left(\mathbb{R}^n\right)$ is a positive periodic function, Caffarelli-Li \cite{cl04} obtained that the entire convex solution to 
\begin{equation}\label{ema}
\det D^2_xu=f
\end{equation}
can be decomposed into a quadratic polynomial and a periodic function. Our result generalizes Caffarelli-Li's work \cite{cl04} from the elliptic case to the parabolic setting. Lu-Li \cite{ll22} weakened $f$'s condition to only positively bounded periodic function. Teixeira-Zhang \cite{tz16} proved that when $f\in C^{1,\alpha}\left(\mathbb{R}^n\right)$ tends to a periodic function at infinity, the solution decays to a quadratic polynomial plus a periodic function. 

As the second main result of the present paper, we consider the existence of $C^{4+\alpha,\frac{4+\alpha}{2}}$ periodic solutions to the parabolic Monge-Ampère type equation
\begin{equation}\label{matype}
\left\{\begin{array}{ll}
\left(1-u_t\right)\det \left(D_x^2u+I\right)=f & \text{in }\mathbb{R}^{n+1}, \\
1-u_t>0, \quad D_x^2u+I>0 & \text{in }\mathbb{R}^{n+1}.
\end{array}\right.
\end{equation}
For $m\in\mathbb{N}$ and $\alpha\in\left(0,1\right]$, we say that 
\[
u\in C^{2m+\alpha,\frac{2m+\alpha}{2}}\left(\mathbb{T}^{n+1}\right)
\] 
if $u\in C^{2m+\alpha,\frac{2m+\alpha}{2}}\left(\mathbb{R}^{n+1}\right)$ satisfies that
\begin{equation}\label{average}
\int_{[0,1]^{n+1}}u(x,t)dxdt=0
\end{equation}
and 
\begin{equation}\label{periodic}
f(x+ke_i,t-l)=f(x,t)
\end{equation}
for every $(x,t)\in\mathbb{R}^{n+1}$, $k,l\in\mathbb{N}$ and $1\le i\le n$.

\begin{theorem}\label{existence}
Let $n\ge1$ and $f\in C^{2+\alpha,\frac{2+\alpha}{2}}\left(\mathbb{R}^{n+1}\right)$ be a positive function satisfying \eqref{periodic}. Then there uniquely exists $u\in C^{4+\alpha,\frac{4+\alpha}{2}}\left(\mathbb{T}^{n+1}\right)$ solving \eqref{matype}, which inherits the same periodicity with $f$, if and only if
\[
-\mkern-19mu\int_{[0,1]^{n+1}}f(x,t)dxdt=1.
\]
\end{theorem}

\begin{remark}\label{general}
Let $\{\epsilon_1,\cdots,\epsilon_n\}$ be a basis of $\mathbb{R}^n$, $\epsilon_0,\tau>0$ be fixed constants, $A$ be an $n\times n$ symmetric positive definite matrix and $f\in C^{2+\alpha,\frac{2+\alpha}{2}}\left(\mathbb{R}^{n+1}\right)$ be a positive function satisfying 
\[
f(x+k\epsilon_i,t+l\epsilon_0)=f(x,t)
\]
for every $(x,t)\in\mathbb{R}^{n+1}$ and $k,l\in\mathbb{N}$. By appropriately modifying the proof of Theorem \ref{existence}, we establish that there uniquely exists a solution $u\in C^{4+\alpha,\frac{4+\alpha}{2}}\left(\mathbb{R}^{n+1}\right)$ that shares the periodicity with $f$, satisfies
\[
-\mkern-19mu\int_{\left\{\left(\sum\limits_{i=1}^n\lambda_i\epsilon_i,\lambda_0\epsilon_0\right)\in\mathbb{R}^{n+1}:0\le\lambda_j\le1,0\le j\le n\right\}}u(x,t)dxdt=0
\]
and solves
\[
\left\{\begin{array}{ll}
\left(\tau-u_t\right)\det \left(D_x^2u+A\right)=f & \text{in }\mathbb{R}^{n+1}, \\
\tau-u_t>0, \quad D_x^2u+A>0 & \text{in }\mathbb{R}^{n+1},
\end{array}\right.
\]
if and only if the compatibility condition \eqref{coefficient} holds.
\end{remark}

\begin{remark}
Theorem \ref{existence} represents a parabolic extension of Theorem 2.2 in \cite{l90}. There Li \cite{l90} gave a necessary and sufficient condition for the existence of periodic solution to a Monge-Ampère type equation. More precisely, let $A$ be an $n\times n$ positive definite matrix and $f\in C^{2+\alpha}\left(\mathbb{R}^n\right)$ be a positive function satisfying
\[
f(x+e_i)=f(x)
\]
for every $x\in\mathbb{R}^n$, and $1\le i\le n$. Then there exists a unique solution $C^{4+\alpha}\left(\mathbb{R}^n\right)$ that has zero average in $[0,1]^n$ and solves
\[
\left\{\begin{array}{ll}
\det \left(A+D_x^2u\right)=f & \text{in }\mathbb{R}^n, \\
A+D_x^2u>0 & \text{in }\mathbb{R}^n,
\end{array}\right.
\]
if and only if 
\[
\det A=-\mkern-19mu\int_{[0,1]^n}f(x)dx.
\]
\end{remark}

Caffarelli-Li \cite{cl03} derived the existence of the entire convex solution of \eqref{ema} for bounded supports of $f-1$ for $n\ge3$. In addition, this solution converges to a convex quadratic polynomial at rate $|x|^{2-n}$ at infinity. This foundational result was subsequently extended by Bao-Li-Zhang \cite{blz15} to the more general case \eqref{beta} in $n\ge3$. For $n=2$, Bao-Xiong-Zhou \cite{bxz19} proved the existence of entire solutions converging to a quadratic polynomial plus a logarithmic term.

Recently, significant progress has been made in the existence of the ancient solution to \eqref{pma1} in $\mathbb{R}^{n+1}_-$. Two fundamental viscosity solutions when $f\equiv1$ are:
\[
u_0(x,t)=-t+\frac{1}{2}|x|^2,\quad u_1(x,t)=\left(\frac{\left(n+1\right)^{n+2}}{\left(2n\right)^n\left(n-1\right)}\right)^{\frac{1}{n+1}}\left(-t\right)^{\frac{1}{n+1}}|x|^{\frac{2n}{n+1}},
\]
where $u_1$ exists only for $n\ge2$. Notably, $u_1$ was firstly introduced by Gutiérrez-Huang \cite{gh98}. Zhou-Gong-Bao \cite{zgb21} established the existence of the solution that tends to $u_0$ at infinity for $n\ge3$, under the condition that $f$ satisfies
\begin{equation}\label{beta}
f(x,t)=1+O\left(u_0\right)^{-\frac{\beta}{2}},\quad\text{as}\quad |x|^2-t\rightarrow\infty,
\end{equation}
where $\beta>2$. An-Bao-Liu \cite{abl24} constructed a smooth solution in variable separated form, that converges to $u_1$ as $|x|\rightarrow\infty$. For $n\ge5$ and $\beta>2$, Bao-Liu-Zhou \cite{blz25} demonstrated the existence of solutions converging to $u_1$ when
\[
f(x,t)=1+O\left(u_1\right)^{-\frac{\beta}{2}},\quad\text{as}\quad |x|^2-t\rightarrow\infty.
\] 

This paper is organized as follows. At the end of Section \ref{introduction}, we introduce essential notations. Section \ref{calculas} presents a calculas lemma related to the periodicity. Theorems \ref{existence} is proved in Sections \ref{thm1}. While the proof of Theorem \ref{liouville} is divided into two parts: Sections \ref{thm2part1} establishes the asymptotic behavior $u \rightarrow -\tau t + \frac{1}{2}x'Ax$ at infinity using the nonlinear perturbation method from~\cite{zbw18}, which ensures bounded Hessian and uniform parabolicity;  Section~\ref{thm2part2} analyzes the periodic difference quotient of $u$ in both $x$ and $t$ variables, proving that $u - (-\tau t + \frac{1}{2}x'Ax)$ reduces to a constant after subtracting a linear function in $x$ and a periodic function from Theorem~\ref{existence}.

In this article, we denote $D_x^i D_t^j u$ as $u$'s mixed derivative, which is $i\text{-th}$ order in $x$ and $j\text{-th}$ order in $t$. For $m\in\mathbb{N}$, $u\in C^{2m, m}$ if $D_x^i D_t^j f$ is continous for $i,j\in\mathbb{N}$ and $i+2 j \leq 2m$. For $\alpha\in(0,1)$ and bounded domain $\mathcal{D}\subset\mathbb{R}^{n+1}$, we say $u\in C^{2m+\alpha,\frac{2m+\alpha}{2}}\left(\overline{\mathcal{D}}\right)$ if $u\in C^{2m,m}\left(\overline{\mathcal{D}}\right)$ and
\[
\sup_{\substack{i,j\in\mathbb{N}\\i+2j=2m}}\sup_{\substack{\left(x_1,t_1\right),\left(x_2,t_2\right)\in\mathcal{D} \\ \left(x_1,t_1\right)\ne\left(x_2,t_2\right)}}\frac{\left|D_x^iD_t^ju\left(x_1,t_1\right)-D_x^iD_t^ju\left(x_2,t_2\right)\right|}{\left|\left(x_1,t_1\right)-\left(x_2,t_2\right)\right|_p^{\alpha}}<\infty,
\]
where the parabolic distance between $X=(x,t),Y=(y,s)\in \mathbb{R}^{n+1}$ is denoted by 
\[
\left|X-Y\right|_p:=\left(|x-y|^2+|t-s|\right)^{\frac{1}{2}}. 
\] 
It is easy to verify that $C^{2m+\alpha,\frac{2m+\alpha}{2}}\left(\mathbb{T}^{n+1}\right)$, equipped with the norm 
\[
\left\|u\right\|_{C^{2m+\alpha,\frac{2m+\alpha}{2}}\left([0,1]^{n+1}\right)}, 
\]
is a Banach space. We adopt the following shorthand for partial derivatives:
\[
u_i:=D_{x_i}u,\quad u_{ij}:=D_{x_ix_j}u,\quad u_{ijk}:=D_{x_ix_jx_k}u,\quad u_{ijkl}:=D_{x_ix_jx_kx_l}u,
\]
\[
u_{it}:=D_{x_i}D_tu,\quad u_{ijt}:=D_{x_ix_j}D_tu.
\]

A subset $\mathcal{D}$ in $\mathbb{R}^{n+1}$ is called a bowl-shaped domain if
\[
\mathcal{D}(t):=\left\{x \in \mathbb{R}^n:(x, t) \in \mathcal{D}\right\}
\] 
are convex for each $t\le0$ and $\mathcal{D}\left(t_1\right) \subset \mathcal{D}\left(t_2\right)$ whenever $t_1 \leq t_2$ with $D\left(t_2\right)\neq\emptyset$. The parabolic boundary of a bounded bowl-shaped domain $\mathcal{D}$ is defined as
\[
\partial_p \mathcal{D}:=\left(\overline{\mathcal{D}(t_0)}\times\{t_0\}\right)\cup  \bigcup_{t \le0}(\partial \left(\mathcal{D}(t)\right) \times\{t\}),
\]
where $t_0:=\inf\{t\le0:(x,t)\in \mathcal{D}\}$. A canonical example is the parabolic ball: 
\[
E_{\delta}(x_0,t_0):=\{(x,t)\in\mathbb{R}^{n+1}_-:|x-x_0|^2-(t-t_0)< \delta^2,\quad t\le t_0 \},\quad \delta>0.
\]
We abbreviate $E_{\delta}:=E_{\delta}\left(0_n,0\right)$, where $0_n$ denotes the origin in $\mathbb{R}^n$. 

The parabolic distance from $Z\in \mathbb{R}^{n+1}_-$ to a set $\mathcal{D}\subset\mathbb{R}^{n+1}_-$ is defined by
\[
\text{dist}_p\left(Z,\mathcal{D}\right):=\inf\left\{\left|Z-Z_1\right|_p: Z_1\in\mathcal{D} \right\}.
\]

\quad

\section{a calculus lemma related to the periodicity}\label{calculas}

\begin{lemma}\label{periodicity lemma}
Let $u\in C^{2,1}\left(\mathbb{R}^{n+1}\right)$ satisfy \eqref{periodic}. Then
\[
\int_{[0,1]^{n+1}}\left(1-u_t\right)\det \left(I+D^2_xu\right)dxdt=1.
\]
\end{lemma}

\begin{proof}
We prove the indentity for smooth $u$. Then the case for $u\in C^{2,1}\left(\mathbb{R}^{n+1}\right)$ can be obtained by approximation. For $n=1$, it follows that by direct calculations
\[
\int_{[0,1]^2}\left(1-u_t\right)\left(1+u_{xx}\right)dxdt=1+\int_{[0,1]^2}\left(u_{xx}-u_t-u_tu_{xx}\right)dxdt.
\]
Applying the Divergence Theorem and the periodicity of $u,u_x,u_t$ gives that
\[
\int_{[0,1]^2}u_tdxdt=\int_{[0,1]^2}u_{xx}dxdt=0
\]
and
\[
\begin{aligned}
\int_{[0,1]^2}u_tu_{xx}dxdt&=\int_0^1\left(u_tu_x\right)(1,t)dt-\int_0^1\left(u_tu_x\right)(0,t)dt-\int_{[0,1]^2}u_{xt}u_xdxdt\\
&=0+\frac{1}{2}\int_{[0,1]^2}\left(u_x^2\right)_tdxdt\\
&=\frac{1}{2}\left(\int_0^1\left(u_x^2\right)(x,1)dx-\int_0^1\left(u_x^2\right)(x,0)dx\right)=0,
\end{aligned}
\]
implying that
\[
\int_{[0,1]^2}\left(1-u_t\right)\left(1+u_{xx}\right)dxdt=1.
\]

For $n\ge2$, denote 
\[
w:=u+\frac{1}{2}|x|^2-t,\quad v:=x\cdot D_xw-w-t.
\]
Then $\left(1-u_t\right)\det \left(I+D^2_xu\right)=-w_t\det D_x^2w$ and by direct computations
\[
\begin{aligned}
-w_t\det D_x^2w
&=\frac{\partial\left(D_xw,x\cdot D_xw-w\right)}{\partial\left(x,t\right)}\\
&=\frac{\partial\left(D_xw,v\right)}{\partial\left(x,t\right)}+\frac{\partial\left(D_xw,t\right)}{\partial\left(x,t\right)}\\
&=\sum\limits_{i=1}^n(-1)^{n+1+i}\left(v\frac{\partial\left(D_xw\right)}{\partial\left(\tilde{x_i},t\right)}\right)_i+\left(v\det D_x^2w\right)_t\\
&\quad- v\left(\sum\limits_{i=1}^n(-1)^{n+1+i}\left(\frac{\partial\left(D_xw\right)}{\partial\left(\tilde{x_i},t\right)}\right)_i+\left(\det D_x^2w\right)_t\right)+\frac{\partial\left(D_xw,t\right)}{\partial\left(x,t\right)},
\end{aligned}
\]
where
\[
D_xw:=\left(w_1,\cdots,w_n\right)
\]
and for a vector $x\in\mathbb{R}^n$ and $1\le i\le n$, $\tilde{x_i}$ is defined by
\[
\tilde{x_i}:=\left(x_1,\cdots,x_{i-1},x_{i+1}\cdots,x_n\right)\in\mathbb{R}^{n-1}.
\]
By the lemma of divergence-free rows in Chapter 8 of \cite{e10}
\[
\sum\limits_{i=1}^n(-1)^{n+1+i}\left(\frac{\partial\left(D_xw\right)}{\partial\left(\tilde{x_i},t\right)}\right)_i+\left(\det D_x^2w\right)_t=0,
\]
impling that
\[
\left(1-u_t\right)\det \left(I+D^2_xu\right)=\sum\limits_{i=1}^n(-1)^{n+1+i}\left(v\frac{\partial\left(D_xw\right)}{\partial\left(\tilde{x_i},t\right)}\right)_i+\left(v\det D_x^2w\right)_t+\frac{\partial\left(D_xw,t\right)}{\partial\left(x,t\right)}.
\]

It follows from the Divergence Theorem that
\begin{equation}\label{div1}
\begin{aligned}
\int_{[0,1]^{n+1}}\left(v\det D_x^2w\right)_tdxdt&=\int_{[0,1]^{n+1}}\left(v\det D_x^2w\right)_tdxdt\\
&=\int_{[0,1]^n_{t=1}}v\det D_x^2wdx-\int_{[0,1]^n_{t=0}}v\det D_x^2wdx
\end{aligned}
\end{equation}
and for $1\le i\le n$
\begin{equation}\label{div2}
\int_{[0,1]^{n+1}}\left(v\frac{\partial\left(D_xw\right)}{\partial\left(\tilde{x_i},t\right)}\right)_idxdt=\int_{[0,1]^n_{x_i=1}}v\frac{\partial\left(D_xw\right)}{\partial\left(\tilde{x_i},t\right)}d\tilde{x_i}dt-\int_{[0,1]^n_{x_i=0}}v\frac{\partial\left(D_xw\right)}{\partial\left(\tilde{x_i},t\right)}d\tilde{x_i}dt,
\end{equation}
where for $a\in[0,1]$
\[
[0,1]^n_{t=a}:=[0,1]^{n+1}\cap\left\{(x,t)\in\mathbb{R}^{n+1}:t=a\right\},
\] 
\[
[0,1]^n_{x_i=a}:=[0,1]^{n+1}\cap\left\{(x,t)\in\mathbb{R}^{n+1}:x_i=a\right\},\quad 1\le i\le n.
\]
Combined with the periodicity of $v\det D_x^2w$, we derive from \eqref{div1} that
\[
\int_{[0,1]^{n+1}}\left(v\det D_x^2w\right)_tdxdt=0.
\]
We notice that
\[
v=\left(x_iw_i-\frac{1}{2}x_i^2\right)+\left(\tilde{x_i}\cdot\tilde{w_i}-u-\frac{1}{2}\left|\tilde{x_i}\right|^2\right).
\]
When we fix $\tilde{x_i}$ and change $x_i$ from $0$ to $1$, the first term varies from $0$ to $w_i-\frac{1}{2}=u_i+\frac{1}{2}$ and the second term preserves its value due to the periodicity of $u$ and $D_xu,D_x^2u$. Thus from \eqref{div2}
\begin{equation}\label{n+1ton}
\int_{[0,1]^{n+1}}\left(v\frac{\partial\left(D_xw\right)}{\partial\left(\tilde{x_i},t\right)}\right)_idxdt=\frac{1}{2}\int_{[0,1]^n_{x_i=1}}\left(2u_i+1\right)\frac{\partial\left(D_xw\right)}{\partial\left(\tilde{x_i},t\right)}d\tilde{x_i}dt.
\end{equation}
Applying the expansion for the determinent and the lemma of divergence-free rows
\begin{equation}\label{dfs1}
\begin{aligned}
\left(2u_i+1\right)\frac{\partial\left(D_xw\right)}{\partial\left(\tilde{x_i},t\right)}
&=\frac{\partial\left(w_1,\cdots,u_i^2+u_i,\cdots,w_n\right)}{\partial\left(\tilde{x_i},t\right)}\\
&=\sum\limits_{j<i}(-1)^{i+j}\left(\left(u_i^2+u_i\right)\frac{\partial\left(\tilde{w_i}\right)}{\partial\left(\tilde{x_{j,i}},t\right)}\right)_j\\
&\quad-\sum\limits_{n\ge j>i}(-1)^{i+j}\left(\left(u_i^2+u_i\right)\frac{\partial\left(\tilde{w_i}\right)}{\partial\left(\tilde{x_{i,j}},t\right)}\right)_j\\
&\quad+(-1)^{i+n}\left(\left(u_i^2+u_i\right)\frac{\partial\left(\tilde{w_i}\right)}{\partial\left(\tilde{x_i}\right)}\right)_t,
\end{aligned}
\end{equation}
where for a vector $x\in\mathbb{R}^n$ and $1\le i<j\le n$, $\tilde{x_{i,j}}$ is defined by
\[
\tilde{x_{i,j}}:=\left(x_1,\cdots,x_{i-1},x_{i+1}\cdots,x_{j-1},x_{j+1},\cdots,x_n\right)\in\mathbb{R}^{n-2}.
\]
The Divergence Theorem gives that for $j<i$
\begin{equation}\label{dfs2}
\begin{aligned}
\int_{[0,1]^n_{x_i=1}}\left(\left(u_i^2+u_i\right)\frac{\partial\left(\tilde{w_i}\right)}{\partial\left(\tilde{x_{j,i}},t\right)}\right)_jd\tilde{x_i}dt
&=\int_{[0,1]^{n-1}_{x_i,x_j=1}}\left(u_i^2+u_i\right)\frac{\partial\left(\tilde{w_i}\right)}{\partial\left(\tilde{x_{j,i}},t\right)}d\tilde{x_{j,i}}dt\\
&\quad-\int_{[0,1]^{n-1}_{x_i=1,x_j=0}}\left(u_i^2+u_i\right)\frac{\partial\left(\tilde{w_i}\right)}{\partial\left(\tilde{x_{j,i}},t\right)}d\tilde{x_{j,i}}dt,
\end{aligned}
\end{equation}
where for $a,b\in[0,1]$
\[
[0,1]^{n-1}_{x_i=a,x_j=b}:=[0,1]^{n+1}\cap\left\{(x,t)\in\mathbb{R}^{n+1}:x_i=a,x_j=b\right\}.
\]
It follows from the periodicity of $u_i^2+u_i$ that for $j<i$
\[
\int_{[0,1]^n_{x_i=1}}\left(\left(u_i^2+u_i\right)\frac{\partial\left(\tilde{w_i}\right)}{\partial\left(\tilde{x_{j,i}},t\right)}\right)_jd\tilde{x_i}dt=0.
\]
Similarly we calculate the left two terms in \eqref{dfs1} and then derive that for $j>i$
\[
\int_{[0,1]^n_{x_i=1}}\left(\left(u_i^2+u_i\right)\frac{\partial\left(\tilde{w_i}\right)}{\partial\left(\tilde{x_{i,j}},t\right)}\right)_jd\tilde{x_i}dt=0
\]
and
\[
\int_{[0,1]^n_{x_i=1}}\left(\left(u_i^2+u_i\right)\frac{\partial\left(\tilde{w_i}\right)}{\partial\left(\tilde{x_i}\right)}\right)_td\tilde{x_i}dt=0.
\]
Therefore by \eqref{n+1ton} and \eqref{dfs1}, we obtain that
\[
\int_{[0,1]^{n+1}}\left(v\frac{\partial\left(D_xw\right)}{\partial\left(\tilde{x_i},t\right)}\right)_idxdt=0.
\]

Utill now we arrive at
\[
\int_{[0,1]^{n+1}}\left(1-u_t\right)\det \left(I+D^2_xu\right)dxdt=\int_{[0,1]^{n+1}}\frac{\partial\left(D_xw,t\right)}{\partial\left(x,t\right)}dxdt.
\]
Direct calculations give that for $n\ge3$
\begin{equation}\label{onebyone}
\begin{aligned}
\frac{\partial\left(D_xw,t\right)}{\partial\left(x,t\right)}
&=\frac{\partial\left(\hat{x_1},\check{w_2},t\right)}{\partial\left(x,t\right)}+\frac{\partial\left(u_1,\check{w_2},t\right)}{\partial\left(x,t\right)}\\
&=\frac{\partial\left(\hat{x_2},\check{w_3},t\right)}{\partial\left(x,t\right)}+\frac{\partial\left(u_1,\check{w_2},t\right)}{\partial\left(x,t\right)}+\frac{\partial\left(\check{x_1},u_2,\check{w_3},t\right)}{\partial\left(x,t\right)}\\
&=\cdots\\
&=\frac{\partial\left(x,t\right)}{\partial\left(x,t\right)}+\frac{\partial\left(u_1,\check{w_2},t\right)}{\partial\left(x,t\right)}+\sum_{i=2}^{n-1}\frac{\partial\left(\hat{x_{i-1}},u_i,\check{w_{i+1}}\right)}{\partial\left(x,t\right)}+\frac{\partial\left(\hat{x_{n-1}},u_n,t\right)}{\partial\left(x,t\right)}\\
&=1+\frac{\partial\left(u_1,\check{w_2},t\right)}{\partial\left(x,t\right)}+\sum_{i=2}^{n-1}\frac{\partial\left(\hat{x_{i-1}},u_i,\check{w_{i+1}}\right)}{\partial\left(x,t\right)}+\frac{\partial\left(\hat{x_{n-1}},u_n,t\right)}{\partial\left(x,t\right)},
\end{aligned}
\end{equation}
where for a vector $x\in\mathbb{R}^n$ and $1\le i\le n$, we define
\[
\hat{x_i}:=\left(x_1,\cdots,x_i\right)\in\mathbb{R}^i,\quad \check{x_i}:=\left(x_i,\cdots,x_n\right)\in\mathbb{R}^{n-i}.
\]
It follows that for every $2\le i\le n-1$
\[
\begin{aligned}
&\int_{[0,1]^{n+1}}\frac{\partial\left(\hat{x_{i-1}},u_i,\check{w_{i+1}}\right)}{\partial\left(x,t\right)}dxdt\\
&=\int_{[0,1]^{n+1}}\left(\sum_{j=1}^n(-1)^{i+j}u_{ij}\frac{\partial\left(\hat{x_{i-1}},\check{w_{i+1}}\right)}{\partial\left(\tilde{x_i},t\right)}-(-1)^{i+n}u_{it}\frac{\partial\left(\hat{x_{i-1}},\check{w_{i+1}}\right)}{\partial\left(x\right)}\right)dxdt\\
&=\int_{[0,1]^{n+1}}\left(\sum_{j=1}^n(-1)^{i+j}\left(u_i\frac{\partial\left(\hat{x_{i-1}},\check{w_{i+1}}\right)}{\partial\left(\tilde{x_i},t\right)}\right)_j-(-1)^{i+n}\left(u_i\frac{\partial\left(\hat{x_{i-1}},\check{w_{i+1}}\right)}{\partial\left(x\right)}\right)_t\right)dxdt\\
&=\sum_{j=1}^n(-1)^{i+j}\int_{[0,1]^n_{x_j=1}}u_i\frac{\partial\left(\hat{x_{i-1}},\check{w_{i+1}}\right)}{\partial\left(\tilde{x_i},t\right)}d\hat{x_j}dt-(-1)^{i+n}\int_{[0,1]^n_{t=1}}u_i\frac{\partial\left(\hat{x_{i-1}},\check{w_{i+1}}\right)}{\partial\left(x\right)}dx\\
&\quad-\sum_{j=1}^n(-1)^{i+j}\int_{[0,1]^n_{x_j=0}}u_i\frac{\partial\left(\hat{x_{i-1}},\check{w_{i+1}}\right)}{\partial\left(\tilde{x_i},t\right)}d\hat{x_j}dt+(-1)^{i+n}\int_{[0,1]^n_{t=0}}u_i\frac{\partial\left(\hat{x_{i-1}},\check{w_{i+1}}\right)}{\partial\left(x\right)}dx\\
&=0,
\end{aligned}
\]
where the above four equalities are valid up to down from the expansion for the determinent, the lemma of divergence-free rows, the Divergence Theorem and the periodicity. Similarly there is
\[
\int_{[0,1]^{n+1}}\frac{\partial\left(u_1,\check{w_2},t\right)}{\partial\left(x,t\right)}dxdt=0.
\]
It is obvious that
\[
\int_{[0,1]^{n+1}}\frac{\partial\left(\hat{x_{n-1}},u_n,t\right)}{\partial\left(x,t\right)}dxdt=\int_{[0,1]^{n+1}}u_{nn}dxdt=0.
\]
By \eqref{onebyone}, we arrived at
\[
\int_{[0,1]^{n+1}}\frac{\partial\left(D_xw,t\right)}{\partial\left(x,t\right)}dxdt=1
\]
for $n\ge3$. The case for $n=2$ is much easier. Consequently we attain that
\[
\int_{[0,1]^{n+1}}\left(1-u_t\right)\det \left(I+D^2_xu\right)dxdt=1.
\]
\end{proof}

\section{proof of Theorem \ref{existence}}\label{thm1}

\begin{proof}
The solution's uniqueness is established through the application of the strong maximum principle combined with the inherent periodicity properties.

The necessity follows from Lemma \ref{periodicity lemma}. The main difficulty lies from proving the sufficiency. We denote for $u\in C^{4+\alpha,\frac{4+\alpha}{2}}\left(\mathbb{T}^{n+1}\right)$ and $\sigma\in[0,1]$
\[
\mathcal{S}_\sigma(u):=\left(1-u_t\right)\det \left(D_x^2u+I\right)-\sigma f-(1-\sigma).
\]
By Lemma \ref{periodicity lemma}, $\mathcal{S}_\sigma$ is actually an operator from $C^{4+\alpha,\frac{4+\alpha}{2}}\left(\mathbb{T}^{n+1}\right)$ to $C^{2+\alpha,\frac{2+\alpha}{2}}\left(\mathbb{T}^{n+1}\right)$ for every $\sigma\in[0,1]$. Define
\[
\mathcal{C}:=C^{4+\alpha,\frac{4+\alpha}{2}}\left(\mathbb{T}^{n+1}\right)\cap\left\{u\in C^{2,1}\left(\mathbb{R}^{n+1}\right):1-u_t>0,\quad D_x^2u+I>0\text{ in }\mathbb{R}^{n+1}\right\}
\] 
\[
\mathcal{T}:=\left\{\sigma\in[0,1]: \exists u\in \mathcal{C} \text{ s.t. }\mathcal{S}_\sigma(u)=0\text{ in }\mathbb{R}^{n+1} \right\}.
\]
Obviously there is $0\in\mathcal{T}$. Combined with the following Lemmas \ref{open} and \ref{closed}, $\mathcal{T}$ is open and closed. Thus $\mathcal{T}=[0,1]$. The theorem is proved.
\end{proof}

First we prove an existence result for a heat equation with periodic nonhomogenous term by using the Fourier expansions.

\begin{prop}\label{periodic heat equation}
Let $\alpha\in(0,1)$ and $h\in C^{\alpha,\frac{\alpha}{2}}\left(\mathbb{T}^{n+1}\right)$. Then there uniquely exists a solution $u\in C^{2+\alpha,\frac{2+\alpha}{2}}\left(\mathbb{T}^{n+1}\right)$ to the equation
\begin{equation}\label{nonhomogenous equation}
u_t-\Delta_xu=h\quad \text{in}\quad\mathbb{R}^{n+1}.
\end{equation}
\end{prop}

\begin{proof}
Firstly we solve the equation \eqref{nonhomogenous equation} for $h\in C^{\infty}\left(\mathbb{T}^{n+1}\right)$. Due to its smoothness and Theorem 3.3.6 in \cite{g14}, $h$ has pointwise convergent Fourier expansions
\[
h(x,t)=\sum\limits_{m\in\mathbb{Z}^{n+1}}\hat{h}(m)e^{2\pi i(x,t)\cdot m},
\]
where $m=\left(m_1,\cdots,m_n,m_{n+1}\right)$ and the Fourier coefficients $\hat{h}(m)$ are defined by
\[
\hat{h}(m):=\int_{[0,1]^{n+1}}h(x,t)e^{-2\pi i(x,t)\cdot m}dxdt.
\]
It follows that $\hat{h}(0,\cdots,0,0)=0$. We define
\[
\hat{u}(0,\cdots,0,0)=0
\]
and for $m\ne(0,\cdots,0,0)$
\[
\hat{u}(m):=\frac{\hat{h}(m)}{2\pi im_{n+1}+\sum\limits_{i=1}^n4\pi^2m_i^2},\quad u(x,t):=\sum\limits_{m\in\mathbb{Z}^{n+1}}\hat{u}(m)e^{2\pi i(x,t)\cdot m}.
\]
In order to prove that $u\in C^{\infty}\left(\mathbb{T}^{n+1}\right)$, we need another representation for $\hat{h}(m)$ with decay in $m$. For fixed $m\in\mathbb{Z}^{n+1}\setminus\left\{(0,\cdots,0,0)\right\}$, there exist integers $1\le i_1,\cdots,i_r\le n+1$, such that $m_{i_s}=0$ for $1\le s\le r$ and $m_i\ne0$ for $i\notin\left\{i_1,\cdots,i_r\right\}$. Then for every $k\in \mathbb{N}^+$
\[
\begin{aligned}
\hat{h}(m)&=(-2\pi i)^{\left|\alpha_m^k\right|}\left(m^{\alpha_m^k}\right)^{-1}\int_{[0,1]^n}D_{x,t}^{\alpha_m^k}h(x,t)e^{-2\pi i(x,t)\cdot m}dxdt\\
&=(-2\pi i)^{\left|\alpha_m^k\right|}\left(m^{\alpha_m^k}\right)^{-1}\left(D_{x,t}^{\alpha_m^k}h\right)^{\wedge}(m),
\end{aligned}
\]
where $\alpha_m^k$ is a multiple index defined by
\[
\left(\alpha_m^k\right)_i:=
\left\{\begin{array}{ll}
k & i\notin\left\{i_1,\cdots,i_r\right\}, \\
0 & i\in\left\{i_1,\cdots,i_r\right\},
\end{array}\right.
\]
and
\[
D_{x,t}^{\alpha_m^k}h:=D_x^{\left(\left(\alpha_m^k\right)_1,\cdots,\left(\alpha_m^k\right)_n\right)}D_t^{\left(\alpha_m^k\right)_{n+1}}h.
\]
For any multiple index $\alpha=\left(\alpha_1,\cdots,\alpha_n,\alpha_{n+1}\right)$, we choose $k$ large enough such that the series
\[
\begin{aligned}
&\sum\limits_{m\in\mathbb{Z}^{n+1}}\hat{u}(m)(2\pi i)^{|\alpha|}m^\alpha e^{2\pi i(x,t)\cdot m}\\
&=\sum\limits_{m\in\mathbb{Z}^{n+1}}\frac{\hat{h}(m)}{2\pi im_{n+1}+\sum\limits_{i=1}^n4\pi^2m_i^2}(2\pi i)^{|\alpha|}m^\alpha e^{2\pi i(x,t)\cdot m}\\
&=\sum\limits_{m\in\mathbb{Z}^{n+1}}\frac{(-2\pi i)^{\left|\alpha_m^k\right|}\left(m^{\alpha_m^k}\right)^{-1}\left(D_{x,t}^{\alpha_m^k}h\right)^{\wedge}(m)}{2\pi im_{n+1}+\sum\limits_{i=1}^n4\pi^2m_i^2}(2\pi i)^{|\alpha|}m^\alpha e^{2\pi i(x,t)\cdot m}
\end{aligned}
\]
are absolutely uniformly convergent, where the Fourier coefficients $\left(D_{x,t}^{\alpha_m^k}h\right)^{\wedge}(m)$ are uniformly bounded by the Plancherel's indentity (see Proposition 3.1.16 in \cite{g14}) 
\[
\sum_{m\in\mathbb{Z}^{n+1}}\left|\left(D_{x,t}^{\alpha_m^k}h\right)^{\wedge}(m)\right|^2=\int_{[0,1]^{n+1}}\left|D_{x,t}^{\alpha_m^k}h\right|^2dxdt<\infty.
\]
Thus
\[
D^\alpha_{x,t} u(x,t)=\sum\limits_{m\in\mathbb{Z}^{n+1}}\hat{u}(m)(2\pi i)^{|\alpha|}m^\alpha e^{2\pi i(x,t)\cdot m}.
\]
By the arbitrariness of $\alpha$, $u\in C^{\infty}\left(\mathbb{T}^{n+1}\right)$. Additionally $u_t-\Delta_xu=h$ in $\mathbb{R}^{n+1}$. 

To solve the equation for $h\in C^{\alpha,\frac{\alpha}{2}}\left(\mathbb{T}^{n+1}\right)$, we use the approximation argument. Actually, there exist $\left\{h_{(k)}\right\}_{k=1}^\infty\subset C^{\infty}\left(\mathbb{T}^{n+1}\right)$, such that
\begin{equation}\label{approx}
\left\|h_{(k)}-h\right\|_{C^{\alpha,\frac{\alpha}{2}}\left(\mathbb{T}^{n+1}\right)}\rightarrow0\quad\text{as}\quad k\rightarrow\infty.
\end{equation}
For each $k$, there exists $u_{(k)}\in C^{\infty}\left(\mathbb{T}^{n+1}\right)$, such that $D_tu_{(k)}-\Delta_xu_{(k)}=h_{(k)}$ in $\mathbb{R}^{n+1}$. The Schauder estimates (see Theorem 1 of Chapter 4 in \cite{f64}) give that
\[
\left\|u_{(k)}\right\|_{C^{2+\alpha,\frac{2+\alpha}{2}}\left(\mathbb{T}^{n+1}\right)}\le C\left(\left\|h_{(k)}\right\|_{C^{\alpha,\frac{\alpha}{2}}\left(\mathbb{T}^{n+1}\right)} + \left\|u_{(k)}\right\|_{C^0\left(\mathbb{T}^{n+1}\right)} \right).
\]
The Harnack inequality (see Theorem 7.36-7-37 in \cite{l96}) implies that
\[
\sup_{\mathbb{R}^{n+1}}\left(u_{(k)}-\inf_{\mathbb{R}^{n+1}}u_{(k)}\right)\le C\left\|h_{(k)}\right\|_{C^0\left(\mathbb{T}^{n+1}\right)}.
\]
Thus 
\[
\left\|u_{(k)}\right\|_{C^{2+\alpha,\frac{2+\alpha}{2}}\left(\mathbb{T}^{n+1}\right)}\le C\left\|h_{(k)}\right\|_{C^{\alpha,\frac{\alpha}{2}}\left(\mathbb{T}^{n+1}\right)}\le C.
\]
By the Arzela-Ascoli Theorem, there exists $u_{(0)}\in C^{2+\alpha,\frac{2+\alpha}{2}}\left(\mathbb{T}^{n+1}\right)$, such that 
\[
u_{(k)}\rightarrow u_{(0)}\quad \text{in}\quad C^{2,1}\left(\mathbb{T}^{n+1}\right)
\]
as $k\rightarrow\infty$ along a subsequence. Combined with \eqref{approx}, we obtain that
\[
D_tu_{(0)}-\Delta_xu_{(0)}=h\quad \text{in}\quad\mathbb{R}^{n+1}.
\]
\end{proof}

Applying the above Proposition, we attain the openness. 

\begin{lemma}\label{open}
$\mathcal{T}$ is open in $[0,1]$.
\end{lemma}

\begin{proof}
The G$\hat{\text{a}}$teaux derivative of $\mathcal{S}_\sigma$ at $u$ in $v$ direction is that
\[
\begin{aligned}
\mathcal{S}_\sigma'(u)[v]&=\lim\limits_{\epsilon\rightarrow0}\frac{\mathcal{S}_\sigma(u+\epsilon v)-\mathcal{S}_\sigma(u)}{\epsilon}\\
&=\det\left(D_x^2u+I\right)\left(-v_t+\left(1-u_t\right)\left(I+D_x^2u\right)^{ij}D_{ij}v\right),
\end{aligned}
\]
where $v\in C^{4+\alpha,\frac{4+\alpha}{2}}\left(\mathbb{T}^{n+1}\right)$. 

Fix $\hat{\sigma}\in\mathcal{T}$. Then there exists $u_{\hat{\sigma}}\in C^{4+\alpha,\frac{4+\alpha}{2}}\left(\mathbb{T}^{n+1}\right)$ such that $\mathcal{S}_{\hat{\sigma}}\left(u_{\hat{\sigma}}\right)=0$ in $\mathbb{R}^{n+1}$. We claim that $\mathcal{S}_{\hat{\sigma}}'\left(u_{\hat{\sigma}}\right)$ is onto from $C^{4+\alpha,\frac{4+\alpha}{2}}\left(\mathbb{T}^{n+1}\right)$ to $C^{2+\alpha,\frac{2+\alpha}{2}}\left(\mathbb{T}^{n+1}\right)$. 
For this, we need to solve the linear periodic problem 
\[
\mathcal{S}_{\hat{\sigma}}'\left(u_{\hat{\sigma}}\right)[v]=h\quad \text{in} \quad\mathbb{R}^{n+1}
\]
for $h\in C^{2+\alpha,\frac{2+\alpha}{2}}\left(\mathbb{T}^{n+1}\right)$. Define the operator
\[
L_\sigma v:=\left(\left(1-\sigma\right)\mathcal{S}_{\hat{\sigma}}'\left(u_{\hat{\sigma}}\right)+\sigma\left(D_t-\Delta_x\right)\right)v.
\]
$L_\sigma$ is an operator mapping $C^{4+\alpha,\frac{4+\alpha}{2}}\left(\mathbb{T}^{n+1}\right)$ into $C^{2+\alpha,\frac{2+\alpha}{2}}\left(\mathbb{T}^{n+1}\right)$ since
\[
\int_{[0,1]^{n+1}}\left(D_t-\Delta_x\right)vdxdt=0
\] 
and
\[
\begin{aligned}
\int_{[0,1]^{n+1}}\mathcal{S}_{\hat{\sigma}}'\left(u_{\hat{\sigma}}\right)[v]dxdt&=\int_{[0,1]^{n+1}}\frac{d}{ds}\Bigg|_{s=0}\left(\mathcal{S}_{\hat{\sigma}}\left(u_{\hat{\sigma}}+sv\right)\right)dxdt\\
&=\frac{d}{ds}\left(\int_{[0,1]^{n+1}}\mathcal{S}_{\hat{\sigma}}\left(u_{\hat{\sigma}}+sv\right)dxdt\right)\Bigg|_{s=0}=0.
\end{aligned}
\]
The Schauder estimates (see Theorem 1 of Chapter 4 in \cite{f64}) and the Harnack inequality (see Theorem 7.36-7-37 in \cite{l96}) give that
\begin{equation}\label{star}
\left\|v\right\|_{C^{4+\alpha,\frac{4+\alpha}{2}}\left(\mathbb{T}^{n+1}\right)}\le C\left\|L_\sigma v\right\|_{C^{2+\alpha,\frac{2+\alpha}{2}}\left(\mathbb{T}^{n+1}\right)},
\end{equation}
where the constant $C$ is independent of $\sigma\in[0,1]$. Combining Proposition \ref{periodic heat equation} and Theorem 5.2 in \cite{gt01}, we derive that $\mathcal{S}_{\hat{\sigma}}'\left(u_{\hat{\sigma}}\right)$ is onto from $C^{4+\alpha,\frac{4+\alpha}{2}}\left(\mathbb{T}^{n+1}\right)$ to $C^{2+\alpha,\frac{2+\alpha}{2}}\left(\mathbb{T}^{n+1}\right)$. 

The estimate \eqref{star} also implies that $\mathcal{S}_{\hat{\sigma}}'\left(u_{\hat{\sigma}}\right)$ is one-to-one from $C^{4+\alpha,\frac{4+\alpha}{2}}\left(\mathbb{T}^{n+1}\right)$ to $C^{2+\alpha,\frac{2+\alpha}{2}}\left(\mathbb{T}^{n+1}\right)$. Thus
\[
\mathcal{S}_{\hat{\sigma}}'\left(u_{\hat{\sigma}}\right):C^{4+\alpha,\frac{4+\alpha}{2}}\left(\mathbb{T}^{n+1}\right) \rightarrow C^{2+\alpha,\frac{2+\alpha}{2}}\left(\mathbb{T}^{n+1}\right)
\]
is a linear isomorphism. From Theorem 15.1 in \cite{d85}, there exists a neighborhood $\mathcal{N}$ of $\hat{\sigma}$ in $[0,1]$, such that for all $\sigma\in\mathcal{N}$, there is $u_\sigma\in C^{4+\alpha,\frac{4+\alpha}{2}}\left(\mathbb{T}^{n+1}\right)$, such that $\mathcal{S}_\sigma\left(u_\sigma\right)=0$ in $\mathbb{R}^{n+1}$. Additionally, from the proof of Theorem 15.1 in \cite{d85} there is
\[
\left\|u_\sigma-u_{\hat{\sigma}}\right\|_{C^{4+\alpha,\frac{4+\alpha}{2}}\left(\mathbb{T}^{n+1}\right)}\le C\left\|\left(\mathcal{S}_{\hat{\sigma}}'\left(u_{\hat{\sigma}}\right)\right)^{-1}\left(f-1\right)\right\|_{C^{4+\alpha,\frac{4+\alpha}{2}}\left(\mathbb{T}^{n+1}\right)}\cdot \left|\sigma-\hat{\sigma}\right|,
\]
where $C$ is a universal constant. Choosing $\mathcal{N}$ small enough such that $\left|\sigma-\hat{\sigma}\right|<<1$, we get that for every $\sigma\in\mathcal{N}$
\[
1-D_tu_\sigma>0,\quad I+D_x^2u_\sigma>0\quad \text{in}\quad\mathbb{R}^{n+1}.
\]
Therefore $\mathcal{N}\subset\mathcal{T}$ and then $\mathcal{T}$ is open.
\end{proof}

Now we obtain the $C^{2,1}$ estimates for the $C^{4,2}$ solution of the equation.

\begin{prop}\label{c21estimates}
Let $u\in C^{4,2}\left(\mathbb{T}^{n+1}\right)$ be the solution to
\[
\left\{\begin{array}{ll}
\left(1-u_t\right)\det\left(I+D_x^2u\right)=f & \text{in }\mathbb{R}^{n+1}, \\
1-u_t>0,\quad I+D_x^2u>0 & \text{in }\mathbb{R}^{n+1},
\end{array}\right.
\]
where $f\in C^{2,1}\left(\mathbb{T}^{n+1}\right)$ is a positive function. Then there exists a constant $C$ depending only on $n,\sup\limits_{\mathbb{R}^{n+1}}f,\inf\limits_{\mathbb{R}^{n+1}}f,\left\|D_xf\right\|_{C^0\left(\mathbb{T}^{n+1}\right)}\left\|D_x^2f\right\|_{C^0\left(\mathbb{T}^{n+1}\right)}$, such that
\[
\left\|u\right\|_{C^{2,1}\left(\mathbb{T}^{n+1}\right)}\le C.
\]
\end{prop}

\begin{proof}
\textbf{Step 1: Estimate of $u$.} Take any $0\le t_2\le t_1\le 2$ and $x\in\mathbb{R}^n$. Combined with $1-u_t>0$, we get that
\[
u\left(x,t_1\right)-u\left(x,t_2\right)=\int_{t_2}^{t_1}u_t(x,s)ds\le t_1-t_2\le2.
\]
Considering the periodicity of $u$, there is for every $x\in\mathbb{R}^n$
\[
\mathop{\text{osc}}\limits_{t\in\mathbb{R}^1}u(x,t)\le2.
\]
$D_x^2u+I>0$ gives that $u_{ii}>-1$ for $1\le i\le n$. Take any $0\le y_1\le z_1\le2$ and $\left(\tilde{x_1},t\right)\in\mathbb{R}^n$, then
\[
u\left(z_1,\tilde{x_1},t\right)-u\left(y_1,\tilde{x_1},t\right)=\int_{y_1}^{z_1}u_1\left(s,x_2,\cdots,x_n,t\right)ds.
\]
By the periodicity of $u$, there exists $\overline{x_1}\in[-1,0]$, such that 
\[
u_1\left(\overline{x_1},\tilde{x_1},t\right)=0.
\]
Thus
\[
\begin{aligned}
u\left(z_1,\tilde{x_1},t\right)-u\left(y_1,\tilde{x_1},t\right)&=\int_{y_1}^{z_1}\int_{\overline{x_1}}^su_{11}\left(\alpha,\tilde{x_1},t\right)d\alpha ds\\
&\ge \int_{y_1}^{z_1}\int_{-1}^s (-1)d\alpha ds\\
&=\int_{y_1}^{z_1}(-s-1)ds\ge-4,
\end{aligned}
\]
implying that for any $\left(\tilde{x_1},t\right)\in\mathbb{R}^n$
\[
\mathop{\text{osc}}\limits_{x_1\in\mathbb{R}^1}u\left(\cdot,\tilde{x_1},t\right)\le4.
\]
Similarly, there is for any $\left(\tilde{x_i},t\right)\in\mathbb{R}^n$
\[
\mathop{\text{osc}}\limits_{x_k\in\mathbb{R}^1}u\left(x_1,\cdots,x_k,\cdots,x_n,t\right)\le4.
\]
Thus we obtain that
\[
\mathop{\text{osc}}\limits_{(x,t)\in\mathbb{R}^{n+1}}u\left(x,t\right)\le4n+2.
\]
It follows from $\int_{[0,1]^{n+1}}u(x,t)dxdt=0$ that there exists $\left(x_0,t_0\right)\in[0,1]^{n+1}$ such that $u\left(x_0,t_0\right)=0$. In conclusion 
\[
\left\|u\right\|_{C^0\left(\mathbb{R}^{n+1}\right)}\le 4n+2.
\]

\textbf{Step 2: Estimate of $D_xu$.} For any $0\le y_1\le z_1\le 2$ and $\left(\tilde{x_1},t\right)\in\mathbb{R}^n$
\[
u_1\left(z_1,\tilde{x_1},t\right)-u_1\left(y_1,\tilde{x_1},t\right)=\int_{y_1}^{z_1}u_{11}\left(s,\tilde{x_1},t\right)ds\ge \int_{y_1}^{z_1}(-1)ds\ge-2.
\]
The periodicity of $u$ implies that there exists $\overline{x_1}\in[0,1]$ such that \[
u_1\left(\overline{x_1},\tilde{x_1},t\right)=0.
\]
Thus 
\[
\left\|u_1\left(\cdot,\tilde{x_1},t\right)\right\|_{C^0\left(\mathbb{R}^1\right)}\le 2.
\]
The arbitrariness of $\left(\tilde{x_1},t\right)$ gives that $\left\|u_1\right\|_{C^0\left(\mathbb{R}^{n+1}\right)}\le 2$. Similarly 
\[
\left\|u_i\right\|_{C^0\left(\mathbb{R}^{n+1}\right)}\le 2
\]
for $2\le i\le n$. Therefore
\[
\left\|D_xu\right\|_{C^0\left(\mathbb{R}^{n+1}\right)}\le 2\sqrt{n}.
\]

\textbf{Step 3: Estimates of $u_t,D_x^2u$.} Define $M_0:=8n+4$ and
\[
\phi(s):=\frac{\left(M_0-s\right)^2}{8M_0^2}.
\]
A direct calculation gives that for $s\in\left[-\frac{M_0}{2},\frac{M_0}{2}\right]$
\[
\phi'(s)=\frac{s-M_0}{4M_0^2}\in\left[-\frac{3}{8M_0},-\frac{1}{8M_0}\right]
\]
and 
\begin{equation}\label{phi'''}
\phi''(s)-\phi'(s)^2\ge0.
\end{equation}
By the periodicity and the regularity of $u$, we assume
\[
\left(\alpha'\left(D_x^2u+I\right)\alpha+1-u_t\right)e^{\phi(u)}
\]
attains its maximum at $\left(x_0,t_0\right)\in[0,1]^{n+1}$ and $\alpha\in \mathbb{S}^{n-1}$. Without loss of generality, we assume that $\alpha=(1,0,\cdots,0)$, $D_x^2u$ is diagnal at $\left(x_0,t_0\right)$ and $u_{11}$ is the maximum diagnal element. Denote $w:=\left(2+u_{11}-u_t\right)e^{\phi(u)}$. Then at $\left(x_0,t_0\right)$, there are
\begin{equation}\label{fermat}
\left(\log w\right)_i=\frac{u_{11i}-u_{ti}}{2+u_{11}-u_t}+\phi'(u)u_i=0,
\end{equation}
\[
\left(\log w\right)_t=\frac{u_{11t}-u_{tt}}{2+u_{11}-u_t}+\phi'(u)u_t=0,
\]
\[
\left(\log w\right)_{ii}=\frac{u_{11ii}-u_{tii}}{2+u_{11}-u_t}-\left(\frac{u_{11i}-u_{ti}}{2+u_{11}-u_t}\right)^2+\phi''(u)u_i^2+\phi'(u)u_{ii}\le0.
\]
Define the linearized operator $L:=\frac{1}{u_t-1}D_t+\left(I+D_x^2u\right)^{ij}D_{ij}$. $L$ is a parabolic operator because $1-u_t>0$ and $D_x^2u+I>0$. Then at $\left(x_0,t_0\right)$
\begin{equation}\label{compute}
\begin{aligned}
0&\ge L\left(\log w\right)\\
&=\frac{1}{u_t-1}\left(\frac{u_{11t}-u_{tt}}{2+u_{11}-u_t}+\phi'(u)u_t\right)\\
&\quad+\sum\limits_{i=1}^n\left(1+u_{ii}\right)^{-1}\left(\frac{u_{11ii}-u_{tii}}{2+u_{11}-u_t}-\left(\frac{u_{11i}-u_{ti}}{2+u_{11}-u_t}\right)^2+\phi''(u)u_i^2+\phi'(u)u_{ii}\right)\\
&=\frac{L\left(u_{11}-u_t\right)}{2u_t+u_{11}-u_t}+\phi'(u)Lu+\left(\phi''(u)-\phi'(u)^2\right)\sum\limits_{i=1}^n\left(1+u_{ii}\right)^{-1}u_i^2\\
&\ge \frac{L\left(u_{11}-u_t\right)}{2+u_{11}-u_t}+\phi'(u)Lu,
\end{aligned}
\end{equation}
where the first and second equalities come from \eqref{fermat} and \eqref{phi'''}. 

From the equation
\[
\log\left(1-u_t\right)+\log\det\left(I+D_x^2u\right)=\log f
\]
and the concavity of the operator $\log\det M$ for symmetric positive definite matrices 
\begin{equation}\label{u11}
L\left(u_{11}\right)\ge\left(\log f\right)_{11},\quad L\left(u_t\right)\ge\left(\log f\right)_t.
\end{equation}
We observe that at $\left(x_0,t_0\right)$
\begin{equation}\label{uequation}
Lu=\frac{u_t}{u_t-1}+\sum\limits_{i=1}^n\frac{u_{ii}}{1+u_{ii}}=n+1-\left(\left(1-u_t\right)^{-1}+\sum\limits\left(1+u_{ii}\right)^{-1}\right)
\end{equation}

It follows from \eqref{compute},\eqref{u11} and \eqref{uequation} that at $\left(x_0,t_0\right)$
\[
\left(1-u_t\right)^{-1}+\sum\limits_{i=1}^n\left(1+u_{ii}\right)^{-1}\le C\left(1+\frac{1}{2+u_{11}-u_t}\right).
\]
Because $1+u_{11}\ge 1+u_{ii}$ at $\left(x_0,t_0\right)$ for $2\le i\le n$
\[
\left(1-u_t\right)^{-1}+\sum\limits_{i=1}^n\left(1+u_{ii}\right)^{-1}\le C\left(1+\frac{1}{1-u_t+\sum\limits_{i=1}^n\left(1+u_{ii}\right)}\right)
\]
If $1-u_t+\sum\limits_{i=1}^n\left(1+u_{ii}\right)\le 1$, we get the desired upper bound. If not, there is
\[
\left(1-u_t\right)^{-1}+\sum\limits_{i=1}^n\left(1+u_{ii}\right)^{-1}\le C.
\]
While
\[
0<\inf\limits_{\mathbb{R}^{n+1}}f\le \left(1-u_t\right)\prod\limits_{i=1}^n\left(1+u_{ii}\right)\le \sup\limits_{\mathbb{R}^{n+1}}f<\infty,
\]
it follows that
\[
1-u_t+\sum\limits_{i=1}^n\left(1+u_{ii}\right)\le C.
\]
Thus we obtain that
\[
\left\|\left(\alpha'\left(D_x^2u+I\right)\alpha+1-u_t\right)e^{\phi(u)}\right\|_{C^0\left(\mathbb{R}^{n+1}\times \mathbb{S}^{n-1}\right)}\le C.
\]
Combined with $1-u_t>0$ and $D_x^2u+I>0$, we obtain that
\[
\left\|u_t\right\|_{C^0\left(\mathbb{T}^{n+1}\right)}+\left\|D_x^2u\right\|_{C^0\left(\mathbb{T}^{n+1}\right)}\le C.
\]
\end{proof}

\begin{remark}\label{upperlower}
Combining the estimates of $\left\|u_t\right\|_{C^0\left(\mathbb{T}^{n+1}\right)},\left\|D_x^2u\right\|_{C^0\left(\mathbb{T}^{n+1}\right)}$ and the equation, there exists a constant $C>1$ depending only on the constants stated in Proposition \ref{c21estimates}, such that
\[
C^{-1}\le 1-u_t\le C,\quad C^{-1}I\le D_x^2u+I\le CI,\quad \text{in}\quad \mathbb{R}^{n+1}.
\]
\end{remark}

The closedness follows from the $C^{2,1}$ estimates.

\begin{lemma}\label{closed}
$\mathcal{T}$ is closed in $[0,1]$.
\end{lemma}

\begin{proof}
Assume there is a sequence $\left\{\sigma_i\right\}_{i=1}^\infty\subset\mathcal{T}$ satifying $\sigma_i\rightarrow\sigma_0$ as $i\rightarrow\infty$. Then there exists a sequence $\left\{u_{(i)}\right\}_{i=1}^\infty\subset C^{4+\alpha,\frac{4+\alpha}{2}}\left(\mathbb{T}^{n+1}\right)$ satisfying
\[
\left(1-D_tu_{(i)}\right)\det\left(I+D_x^2u_{(i)}\right)=\sigma_if+\left(1-\sigma_i\right)\quad \text{in}\quad \mathbb{R}^{n+1}.
\]  
Applying Proposition \ref{c21estimates}, there is a constant $C>0$ independent of $\sigma_i$ such that
\[
\left\|u_{(i)}\right\|_{C^{2,1}\left(\mathbb{T}^{n+1}\right)}\le C.
\]
It follows from the Evans-Krylov estimates (see Theorem 4.13 in \cite{w92}) that 
\begin{equation}\label{ekestimate}
\left\|u_{(i)}\right\|_{C^{2+\alpha,\frac{2+\alpha}{2}}\left(\mathbb{T}^{n+1}\right)}\le C.
\end{equation}
The Schauder estimates (see Theorem 1 of Chapter 4 in \cite{f64}) give that
\[
\left\|u_{(i)}\right\|_{C^{4+\alpha,\frac{4+\alpha}{2}}\left(\mathbb{T}^{n+1}\right)}\le C.
\]
By the Arzela-Ascoli Theorem, there exists $u_{(0)}\in C^{4+\alpha,\frac{4+\alpha}{2}}\left(\mathbb{T}^{n+1}\right)$, such that $u_{(i)}\rightarrow u_{(0)}$ in $C^{4,2}\left(\mathbb{T}^{n+1}\right)$ as $i\rightarrow\infty$ along a subsequence. Then
\[
\left(1-D_tu_{(0)}\right)\det\left(I+D_x^2u_{(0)}\right)=\sigma_0f+\left(1-\sigma_0\right)\quad \text{in}\quad \mathbb{R}^{n+1}.
\]  
It follows from Remark \ref{upperlower} that 
\[
1-D_tu_{(0)}>0,\quad D_x^2u_{(0)}+I>0,\quad\text{in}\quad \mathbb{R}^{n+1}.
\]
Therefore $\sigma_0\in\mathcal{T}$, implying that $\mathcal{T}$ is closed.
\end{proof}

\quad

\section{asymptotic behavior of $u$ at infinity}\label{thm2part1}

Without loss of generality, we assume 
\[
u(0,0)=0,\quad D_xu(0,0)=0.
\]
We name
\[
Q_{H}:=\left\{(x, t) \in \mathbb{R}_{-}^{n+1}: u(x, t)<H\right\},\quad Q_{H}(t):=\left\{x \in \mathbb{R}^{n}: u(x, t)<H\right\}
\]
for $H>0$ and $t\le0$. By Theorem 1.8.2 in \cite{g16}, there exist an affine transformation 
\[
T_{H}x:=a_{H} x+b_{H}
\]
with $\operatorname{det} a_{H}=1$, $R>0$ related to $H$, such that
\[
B_{\alpha_n R}(0) \subset T_{H}\left(Q_{H}(0)\right) \subset B_{R}(0),
\]
where $\alpha_n:=n^{-\frac{3}{2}}$. By Proposition 3.1 in \cite{zbw18}, there exists a constant $C>1$, such that $C^{-1}\le HR^{-2}\le C$. Denote
\[
\Gamma_H(x,t):=\left(\frac{a_Hx}{R},\frac{t}{R^2}\right),\quad Q_H^{*}:=\Gamma_H(Q_H).
\]
Due to Proposition 3.2 in \cite{zbw18}, there exist constants $C_1,C_2>0$ such that 
\begin{equation}\label{location}
E_{C_1}\subset Q_H^{*}\subset E_{C_2}.
\end{equation}
Let
\[
v(y,s):=R^{-2}u\left(\Gamma^{-1}_H(y,s)\right)
\]
and we get that
\begin{equation}\label{vdef}
\left\{\begin{array}{ll}
-v_{s} \operatorname{det} D_y^{2} v=g & \text { in } Q_{H}^{*}, \\
v=HR^{-2}  & \text { on } \partial_{p} Q_{H}^{*}, \\
m_1 \leq -v_{s} \leq m_2 & \text { in } Q_{H}^{*} ,
\end{array}\right.
\end{equation}
where 
\[
g(y,s):=f\left(\Gamma_{H}^{-1}(y,s)\right).
\]
By adapting the techniques in the proof of Theorem 3.2 in \cite{ww92}, there uniquely exists  a p-convex solution $\overline{v}\in C^0(\overline{Q_H^{*}})\cap C^{\infty}(Q_H^{*})$ to the problem
\begin{equation}\label{barvdef}
\left\{\begin{array}{ll}
-\overline{v}_{s} \operatorname{det} D_y^{2} \overline{v}=1 & \text { in } Q_{H}^{*}, \\
\overline{v}=HR^{-2} & \text { on } \partial_{p} Q_{H}^{*}, \\
m_1 \leq -\overline{v}_{s} \leq m_2 & \text { in } Q_{H}^{*} .
\end{array}\right.
\end{equation}

Let 
\begin{equation}\label{epsilondef}
\epsilon_0:=R^{-2},\quad \epsilon_i:=R^{-1}a_He_i,\quad 1\le i\le n.
\end{equation}
Then $\{\epsilon_1,\cdots,\epsilon_n\}$ forms a basis of $\mathbb{R}^n$ and
\begin{equation}\label{cellaverage}
-\mkern-19mu\int_{C_p}g(y,s)dydy=-\mkern-19mu\int_{[0,1]^{n+1}}f(y,s)dydy=1
\end{equation}
where
\[
C_p:=\left\{(y,s)\in\mathbb{R}^{n+1}_-:y=\sum\limits_{i=1}^n\lambda_i\epsilon_i,s=-\lambda_0\epsilon_0,0\le\lambda_i\le1,0\le i\le n\right\}.
\]
In order to establish the homogenization estimate, we construct a periodic corrector function that captures the oscillatory nature of $f$. According to \eqref{vdef}-\eqref{cellaverage} and Remark \ref{general}, given any point $Z\in Q_H^*$, there exists $\xi\in C^{4+\alpha,\frac{4+\alpha}{2}}\left(\mathbb{R}^{n+1}\right)$ satisfying
\begin{equation}\label{periodic corrector}
\left\{\begin{array}{ll}
-\left(\overline{v}_t\left(Z\right)+\xi_t\right)\det\left(D_x^2\overline{v}\left(Z\right)+D_x^2\xi\right)=g & \text { in } \mathbb{R}^{n+1}, \\
\overline{v}_t\left(Z\right)+\xi_t<0, \quad D_x^2\overline{v}\left(Z\right)+D_x^2\xi>0 & \text { in } \mathbb{R}^{n+1}.
\end{array}\right.
\end{equation}
In addition, $\xi$ inherits the same periodicity with $f$ and satisfies that
\begin{equation}\label{average0}
-\mkern-19mu\int_{C_p}\xi(y,s)dyds=0.
\end{equation}

Then there exists a $C^0$ estimate for $\xi$.

\begin{lemma}\label{xilinfty}
There exists a constant $C(n)>0$ such that
\[
\|\xi\|_{C^0\left(\mathbb{R}^{n+1}\right)}\le 4\|D_x^2\overline{v}\left(Z\right)\|\sum_{i=1}^n|\epsilon_i|^2-2\overline{v}_t\left(Z\right)\epsilon_0.
\]
\end{lemma}

\begin{proof}
Define for every $(x,t)\in\mathbb{R}^{n+1}$
\[
h_{(i)}(s):=\xi\left(x+s\epsilon_i,t\right),\quad 1\le i\le n
\]
and 
\[
h_{(0)}(s):=\xi\left(x,t+s\epsilon_0\right),
\]
where $s\in\mathbb{R}^1$. \eqref{periodic corrector} gives the estimates for every $s\in\mathbb{R}^1$
\[
h_{(i)}''(s)\ge -\left\|D_x^2\overline{v}\left(Z\right)\right\|\left|\epsilon_i\right|^2,\quad 1\le i\le n
\]
and
\[
h_{(0)}'(s)\le -D_t\overline{v}\left(Z\right)\left|\epsilon_0\right|.
\]

Take $0\le s_1\le s_2\le 2$, then
\[
h_{(0)}\left(s_2\right)-h_{(0)}\left(s_1\right)=\int_{s_1}^{s_2}h_{(0)}'(s)ds\le -2D_t\overline{v}\left(Z\right)\left|\epsilon_0\right|,
\]
implying that
\begin{equation}\label{tlinfty}
\mathop{\text{osc}}\limits_{t\in\mathbb{R}^1}\xi(x,t)\le -2D_t\overline{v}\left(Z\right)\left|\epsilon_0\right|.
\end{equation}

Take $0\le s_1\le s_2\le 2$, then for every $1\le i\le n$
\[
h_{(i)}\left(s_2\right)-h_{(i)}\left(s_1\right)=\int_{s_1}^{s_2}h_{(i)}'(s)ds.
\]
By the periodicity of $h_{(i)}$, there exists $s_0\in[-1,0]$, such that $h_{(i)}'\left(s_0\right)=0$. Then
\[
h_{(i)}\left(s_2\right)-h_{(i)}\left(s_1\right)=\int_{s_1}^{s_2}\int_{s_0}^sh_{(i)}''\left(\alpha\right)d\alpha ds\ge-4\left\|D_x^2\overline{v}\left(Z\right)\right\|\left|\epsilon_i\right|^2,
\]
implying that
\begin{equation}\label{xlinfty}
\mathop{\text{osc}}\limits_{x\in\mathbb{R}^n}\xi(x,t) \le 4\left\|D_x^2\overline{v}\left(Z\right)\right\|\sum\limits_{i=1}^n\left|\epsilon_i\right|^2.
\end{equation}

It follows from \eqref{average0} that there exists $\left(x_0,t_0\right)\in C_p$ such that $\xi\left(x_0,t_0\right)=0$. Combining \eqref{tlinfty} and \eqref{xlinfty}, we proof the lemma.
\end{proof}

Using $\xi$'s existence and $C^0$ estimate, we derive a $C^0$ estimate of the $v-\overline{v}$.

\begin{prop}\label{homogenization}
There exist constants $C,\beta>0$ independent of $H$, such that
\[
\left\|v-\overline{v}\right\|_{C^0\left(\overline{Q_H^*}\right)}\le C\left(\sum_{i=1}^n|\epsilon_i|^2+\epsilon_0\right)^{\beta},
\]
where $v,\overline{v}$ and $\epsilon_i$ for $0\le i\le n$ are defined in \eqref{vdef}-\eqref{epsilondef}.
\end{prop}

\begin{proof}

Let $M:=\sup\limits_{Q_H^*}|v-\overline{v}|$. We only treat the case $M=\sup\limits_{Q_H^*}\left(v-\overline{v}\right)>0$. Let $\overline{X}=\left(\overline{x},\overline{t}\right)$ be a maximum point of $v-\overline{v}$. By Proposition 4.1(1) in \cite{gh01}
\begin{equation}\label{Mbdd}
M=v\left(\overline{X}\right)-\overline{v}\left(\overline{X}\right)\le-\overline{w}\left(\overline{X}\right)\le C\text{dist}_p\left(\overline{X},\partial_pQ_H^*\right)^{\frac{1}{n+1}}\le C.
\end{equation}

Let 
\[
\zeta(x,t):=v(x,t)+\frac{M}{36C_2^2}\left|x-\overline{x}\right|^2-\frac{M}{9C_2^2}\left(t-\overline{t}\right),
\]
where $C_2>0$ is from \eqref{location}. Then $\left(\zeta-\overline{v}\right)\left(\overline{X}\right)=M$. It follows from \eqref{location} that
\[
\left|\zeta-v\right|\le \frac{2M}{9} \quad\text{in}\quad Q_H^*.
\]
By the boundary condition $v=\overline{v}$ on $\partial_pQ_H^*$, it follows that $\left|\zeta-\overline{v}\right|\le \frac{2M}{9}$ on $\partial_pQ_H^*$. Thus there exists $\tilde{X}=\left(\tilde{x},\tilde{t}\right)\in Q_H^*$ such that
\[
\left(\zeta-\overline{v}\right)\big(\tilde{X}\big)=\sup\limits_{Q_H^*}\left(\zeta-\overline{v}\right)\ge M.
\] 
Additionally we have that
\[
\left(v-\overline{v}\right)\big(\tilde{X}\big)=\left[\left(\zeta-\overline{v}\right)-\left(\zeta-v\right)\right]\big(\tilde{X}\big)\ge M-\frac{2M}{9}=\frac{7M}{9}.
\]
Combined with $\left(w-\overline{w}\right)\le -\overline{w}$ and Proposition 4.1(1) in \cite{gh01}, we derive that
\[
\text{dist}_p\left(\tilde{X},\partial_pQ_H^*\right)\ge\frac{1}{C}M^{n+1}=:\delta.
\]
By Lemma 2.3 in \cite{zb18}, for $X\in E_{\frac{\delta}{2}}\big(\tilde{X}\big)$, the higher order derivatives of $\overline{w}$ satisfy
\begin{equation}\label{lagrange}
\begin{aligned}
&\left|D_x^3\overline{v}\left(X\right)\right|+\left|D_x^2\overline{v}_t\left(X\right)\right|+\left|D_x\overline{v}_t\left(X\right)\right|+\left|\overline{v}_{tt}\left(X\right)\right|\\
&\le C\left( \text{dist}_p\left(X,\partial_pQ_H^*\right)^{-\beta_3}+\text{dist}_p\left(X,\partial_pQ_H^*\right)^{-\beta_4} \right)\\
&\le C\text{dist}_p\left(X,\partial_pQ_H^*\right)^{-\beta_4}\le CM^{-\beta_4(n+1)},
\end{aligned}
\end{equation}
where $\beta_3,\beta_4$ are positive constants.

Define
\[
\eta\left(X\right):=\overline{v}\left(X\right)+\xi\left(X\right)-\frac{M}{36C_2^2}\left|x-\overline{x}\right|^2+\frac{M}{9C_2^2}\left(t-\overline{t}\right)+\frac{M}{144C_2^2}\left|x-\tilde{x}\right|^2-\frac{M}{18C_2^2}\left(t-\tilde{t}\right),
\]
where $\xi$ is the periodic corrector defined in \eqref{periodic corrector} for $Z=\tilde{X}$. Then
\[
\left(v-\eta\right)\left(X\right)=\zeta\left(X\right)-\left( \overline{v}\left(X\right) + \xi\left(X\right)   + \frac{M}{144C_2^2}\left|x-\tilde{x}\right|^2 -\frac{M}{18C_2^2}\left(t-\tilde{t}\right) \right).
\]

Take $\delta_1:=\frac{M^{\beta_4(n+1)+1}}{C_0}$, where $C_0>0$ is a sufficiently large constant to be determined. Owing to \eqref{Mbdd}, we may choose $C_0$ such that 
\[
\delta_1=\frac{M^{n+1}}{C}\cdot \frac{CM^{(\beta_4-1)(n+1)+1}}{C_0}<\frac{\delta}{2}. 
\]
For $X\in E_{\delta_1}\big(\tilde{X}\big)$, the derivative estimates of $\left|D_x^3\overline{v}\right|,\left|D_x^2\overline{v}_t\right|$ in \eqref{lagrange} yield
\begin{equation}\label{analysis}
\begin{aligned}
D_x^2\eta\left(X\right)&=D_x^2\overline{v}\left(X\right)+D_x^2\xi\left(X\right)-\frac{M}{24C_2^2}I\\
&\le D_x^2\overline{v}\big(\tilde{X}\big)+CM^{-\beta_4(n+1)}\left(\left|x-\tilde{x}\right|+\left|t-\tilde{t}\right|\right)I+D_x^2\xi\left(X\right)-\frac{M}{24C_2^2}I\\
&\le D_x^2\overline{v}\big(\tilde{X}\big)+CM^{-\beta_4(n+1)}\left|X-\tilde{X}\right|_pI+D_x^2\xi\left(X\right)-\frac{M}{24C_2^2}.
\end{aligned}
\end{equation}
In the last line of \eqref{analysis}, by further increasing $C_0$, we ensure that
\[
CM^{-\beta_4(n+1)}\left|X-\tilde{X}\right|_p<CM^{-\beta_4(n+1)}\frac{M^{\beta_4(n+1)+1}}{C_0}<\frac{M}{24r^2},
\]
which implies that for $X\in E_{\delta_1}\big(\tilde{X}\big)$
\[
D_x^2\eta\left(X\right)<D_x^2\overline{v}\big(\tilde{X}\big)+D_x^2\xi\left(X\right).
\]
Similar to \eqref{analysis}, by estimates of $\left|D_x\overline{v}_t\right|,\left|\overline{v}_{tt}\right|$ in \eqref{lagrange}, we obtain that for $C_0$ large
\[
\eta_t\left(X\right)\ge \overline{v}_t\big(\tilde{X}\big)-CM^{-\beta_4(n+1)}\left|X-\tilde{X}\right|_p+\xi_t\left(X\right)+\frac{M}{18C_2^2}>\overline{v}_t\big(\tilde{X}\big)+\xi_t\left(X\right).
\]
Then for $X\in E_{\delta_1}\big(\tilde{X}\big)$ satisfying $D_x^2\eta\left(X\right)>0$ and $\eta_t\left(X\right)<0$, we derive the inequality
\begin{equation}\label{contradiction}
\begin{aligned}
\left(-\eta_t\det D_x^2\eta\right)\left(X\right)&<-\left(\overline{v}_t\big(\tilde{X}\big)+\xi_t\left(X\right)\right)\det\left(D_x^2\overline{v}\big(\tilde{X}\big)+D_x^2\xi\left(X\right)\right)\\
&=f\left(X\right)=\left(-v_t\det D_x^2v\right)\left(X\right).
\end{aligned}
\end{equation}

At $\tilde{X}$, the maximum point of $\zeta-\overline{v}$, there is
\[
v-\eta=\left(\zeta-\overline{v}\right)-\xi\ge\left(\zeta-\overline{v}\right)-C\left(\sum_{i=1}^n|\epsilon_i|^2+\epsilon_0\right)M^{-(4n+2)(n+1)}.
\]
For $X\in \partial_pE_{\delta_1}\big(\tilde{X}\big)$,
\[
\left(v-\eta\right)\left(X\right)\le \left(\zeta-\overline{v}\right)\big(\tilde{X}\big) + C\left(\sum_{i=1}^n|\epsilon_i|^2+\epsilon_0\right)M^{-(4n+2)(n+1)}-\frac{M\delta_1}{C}.
\]

If 
\begin{equation}\label{contradict}
2C\left(\sum_{i=1}^n|\epsilon_i|^2+\epsilon_0\right)M^{-(4n+2)(n+1)}< \frac{M^{\beta_4(n+1)+2}}{CC_0},
\end{equation}
then 
$\left(v-\eta\right)\left(X\right)<\left(v-\eta\right)\big(\tilde{X}\big)$ for any $X\in \partial_pE_{\delta_1}\big(\tilde{X}\big)$. Let $X_1\in E_{\delta_1}\big(\tilde{X}\big)$ be an interior maximum point of $v-\eta$, then $0<D_x^2v\le D_x^2\eta$, $0>v_t\ge \eta_t$ and $-\eta_t\det D_x^2\eta\ge -v_t\det D_x^2v$ at $X_1$. This contradicts  \eqref{contradiction}, which implies the converse inequality of \eqref{contradict}. Thus we obtain the desired bound
\[
M\le C\left(\sum_{i=1}^n|\epsilon_i|^2+\epsilon_0\right)^{\frac{1}{\left(\beta_4+4n+2\right)(n+1)+2}}.
\]
\end{proof}

We introduce the set
\[
E:=\left\{k_1e_1+\cdots+k_ne_n:k_1,\cdots,k_n\in\mathbb{Z},k_1^2+\cdots+k_n^2>0\right\}.
\]
Applying the nonlinear perturbation method, we attain the asymptotic behavior. 

\begin{prop}\label{asymptotic}
There exist an $n\times n$ positive definite matrix $A$, $\tau>0$ with $\tau\operatorname{det}T^2=1$ and constants $\epsilon,C>0$, such that
\[
C^{-1}I\le A\le CI, \quad m_1\le \tau\le m_2,
\]
\[
\left| u(x,t)-\frac{1}{2}x'Ax+\tau t \right|\le C\left( |x|^{2}-t \right)^{\frac{2-\epsilon}{2}} \quad\text{for}\quad |x|^2-t\ge C.
\]
\end{prop}

\begin{proof}
For $e\in E$, define the rescaled vector
\begin{equation}\label{tildee}
\tilde{e}:=R^{-1}a_He. 
\end{equation}
By Lemma 2.1 in \cite{cl04}, there exist $\alpha', C>0$, such that 
\[
|\tilde{e}|\le CR^{-\alpha'}|e|. 
\]
Lemma \ref{homogenization} implies that for $R>1$
\[
\left\|v-\overline{v}\right\|_{C^0\left(\overline{Q_H^*}\right)}\le C\left(\sum_{i}\left\|\tilde{e_i}\right\|^2+R^{-2}\right)^\beta \le CR^{-\beta\min\left\{2\alpha',2\right\}}.
\]
Following the proof of Proposition 4.1 in \cite{zb18}, we obtain the result.
\end{proof}

Through a simalar process of proving Proposition 4.5 in \cite{zb18}, we get the following Hessian estimates.

\begin{prop}\label{bdd}
There exists $C>0$, such that $C^{-1}I\le D_x^2u\le CI$ in $\mathbb{R}^{n+1}_-$.
\end{prop}

\section{proof of Theorem \ref{liouville}}\label{thm2part2}

Let $\lambda>0$. Define the parabolic scaling
\[
u^\lambda(x,t):=\lambda^{-2}u(\lambda x,\lambda^2 t),\quad Q(x,t):=-\tau t+\frac{1}{2}x'Ax,
\] 
where the matrix $A$ and constant $\tau$ are as specified in Proposition \ref{asymptotic}. To obtain a refined description of $u$, we establish the following convergence properties of $u^\lambda$.

\begin{lemma}\label{gradient holder}
$u^\lambda\rightarrow Q$ in $C^0_{loc}\left(\mathbb{R}^{n+1}_-\right)$ as $\lambda\rightarrow\infty$. Moreover, for every $\delta\in(0,1)$, $D_xu^\lambda\rightarrow D_xQ$ in $C^{\delta,\frac{\delta}{2}}_{loc}\left(\mathbb{R}^{n+1}_-\right)$ along a subsequence as $\lambda\rightarrow\infty$.
\end{lemma}

\begin{proof}
By Proposition \ref{asymptotic}, we have for $\lambda^2\left(|x|^2-t\right)\ge C$ that
\[
\left|u^\lambda(x,t)-Q(x,t)\right|\le C\lambda^{-\epsilon}\left(|x|^2-t\right)^{\frac{2-\epsilon}{2}}.
\] 
We conclude that $u^\lambda\rightarrow Q$ in $C^0_{loc}\left(\mathbb{R}^{n+1}_-\right)$ as $\lambda\rightarrow\infty$. Proposition \ref{bdd} gives that
\[
\left|D_xu^\lambda(x_1,t)-D_xu^\lambda(x_2,t)\right|\le C|x_1-x_2|
\] 
for any $x_1,x_2\in\mathbb{R}^n$ and $t\in\left(-\infty,0\right]$. Combined with \eqref{ut} and Lemma 3.1 of Chapter 2 in \cite{lsu68}, we have for every $\delta'\in(0,1)$ that
\[
\left|D_xu^\lambda(x,t_1)-D_xu^\lambda(x,t_2)\right|\le C(K)|t_1-t_2|^{\frac{\delta'}{2}}
\]
for every bounded set $K$ in $\mathbb{R}^{n+1}_-$ and $(x,t_1),(x,t_2)\in K$, where $C(K)$ is a positive constant depending on $K$. Thus for every $\delta''\in\left(0,\delta'\right)$, there is
\[
\left|D_xu^\lambda(x_1,t_1)-D_xu^\lambda(x_2,t_2)\right|\le C(K)\left(|x_1-x_2|^2+|t_1-t_2|\right)^{\frac{\delta''}{2}}.
\]
By Arzela-Ascoli theorem,  we get that for every $\delta\in(0,1)$ 
\[
D_xu^\lambda\rightarrow D_xQ \quad \text{in} \quad C^{\delta,\frac{\delta}{2}}_{loc}\left(\mathbb{R}^{n+1}_-\right)
\]
along a subsequence as $\lambda\rightarrow\infty$.
\end{proof}

For $(x,t)\in\mathbb{R}^{n+1}_-$, $k\in\mathbb{N}^+$, we define the spatial 2-nd difference quotient
\[
\Delta_e^2u(x,t):=|e|^{-2}\left(u(x+e,t)+u(x-e,t)-2u(x,t)\right)
\]
and the temporal difference quotient
\[
\Delta_k^tu(x,t):=k^{-1}\left(u(x,t)-u(x,t-k)\right).
\]
We need the following equations for $\Delta_e^2u$ and $\Delta_k^tu$ derived from the periodicity of $u$ and the concavity of the function $F(a,M):=\log a+\log\det M$ for $a>0$ and $n\times n$ symmetric positive definite matrix $M$.

\begin{lemma}
\begin{equation}\label{xsubsolution}
\frac{1}{u_t}D_t\Delta_e^2u+u^{ij}D_{x_ix_j}\Delta_e^2u\ge0.
\end{equation}
\begin{equation}\label{tsubsolution}
\frac{1}{u_t(x,t-k)}D_t\Delta_k^tu+u^{ij}(x,t-k)D_{x_ix_j}\Delta_k^tu\ge0.
\end{equation}
\begin{equation}\label{tsupersolution}
\frac{1}{u_t(x,t)}D_t\Delta_k^tu+u^{ij}(x,t)D_{x_ix_j}\Delta_k^tu\le0.
\end{equation}
\end{lemma}

Introduce the notations
\[
K_-:=B_1\times(-5,-4),\quad K_+:=B_1\times(-1,0).
\]
The following proposition establishes the precise supremum of $\Delta_e^2u$. 

\begin{prop}\label{deltax}
For every $e\in E$, we have the identity
\[
\sup\limits_{\mathbb{R}^{n+1}_-}\Delta_e^2u=\frac{e'Ae}{|e|^2}.
\]
\end{prop}

\begin{proof}
Let $\beta:=\frac{e'Ae}{|e|^2}$ and $\hat{e}:=\lambda^{-1}e$. 
We claim that
\begin{equation}\label{claim1}
\lim_{\lambda\rightarrow\infty}\int_{K_-}\Delta_{\hat{e}}^2u^{\lambda} dxdt=\int_{K_-}\beta dxdt=\beta\left|K_-\right|.
\end{equation}
Actullay, we observe that
\[
\begin{aligned}
&\quad\int_{K_-}\Delta_{\hat{e}}^2u^{\lambda} dxdt\\
&=\frac{1}{\left|\hat{e}\right|^2}\int_{-5}^{-4}\int_{B_1}\int_0^1\frac{d}{ds}\left[u^{\lambda}\left(x+s\hat{e},t\right)+u^{\lambda}\left(x-s\hat{e},t\right)\right]dsdxdt\\
&=\frac{1}{\left|\hat{e}\right|^2}\int_{-5}^{-4}\int_0^1\left( \int_{B_1\left(s\hat{e}\right)\setminus B_1\left(-s\hat{e}\right)}D_xu^{\lambda}\cdot\hat{e}dx - \int_{B_1\left(-s\hat{e}\right)\setminus B_1\left(s\hat{e}\right)}D_xu^{\lambda}\cdot\hat{e}dx \right)dsdt.
\end{aligned}
\]
Similarly we can get that
\[
\begin{aligned}
&\quad\int_{K_-}\beta dxdt=\int_{K_-}\Delta^2_{\hat{e}}Qdxdt\\
&=\frac{1}{\left|\hat{e}\right|^2}\int_{-5}^{-4}\int_0^1\left( \int_{B_1\left(s\hat{e}\right)\setminus B_1\left(-s\hat{e}\right)}D_xQ\cdot\hat{e}dx - \int_{B_1\left(-s\hat{e}\right)\setminus B_1\left(s\hat{e}\right)}D_xQ\cdot\hat{e}dx \right)dsdt.
\end{aligned}
\]
Lemma \ref{gradient holder} yields the convergence
\[
\begin{aligned}
&\left|\int_{K_-}\Delta_{\hat{e}}^2u^{\lambda} dxdt - \int_{K_-}\beta dxdt \right|\\
&\le\frac{1}{|\hat{e}|}\int_{-5}^{-4}\int_0^1  \int_{B_1\left(s\hat{e}\right)\setminus B_1\left(-s\hat{e}\right)}\left|D_xu^{\lambda}-D_xQ\right|dxdsdt\\
&\quad + \frac{1}{|\hat{e}|}\int_{-5}^{-4}\int_0^1  \int_{B_1\left(-s\hat{e}\right)\setminus B_1\left(s\hat{e}\right)}\left|D_xu^{\lambda}-D_xQ\right|dxdsdt\rightarrow0
\end{aligned}
\]
as $\lambda\rightarrow\infty$. We prove the claim. 

Denote 
\[
\alpha:=\sup\limits_{\mathbb{R}^{n+1}_-}\Delta_e^2u.
\]
We have that $\alpha<\infty$ by the proof of Proposition \ref{bdd}. The strict convexity of $u$ implies that $
\Delta_{\hat{e}}^2u^\lambda>0$. Hence it follows from \eqref{claim1} that $\alpha\ge\beta$. For contradiction, assume $\alpha>\beta$. \eqref{claim1} implies the measure estimate
\[
\limsup_{\lambda\rightarrow\infty}\left( \frac{\alpha+\beta}{2}\left|\left\{ \Delta_{\hat{e}}^2u^{\lambda}\ge \frac{\alpha+\beta}{2} \right\}\cap K_- \right| \right)\le\lim_{\lambda\rightarrow\infty}\int_{K_-}\Delta_{\hat{e}}^2u^{\lambda} dxdt=\beta\left|K_-\right|.
\]
Therefore for $\lambda$ large
\[
\frac{\left|\left\{ \Delta_{\hat{e}}^2u^{\lambda}\ge \frac{\alpha+\beta}{2} \right\}\cap K_- \right|}{\left|K_-\right|}\le\frac{2\beta}{\alpha+\beta}.
\]
Applying \eqref{xsubsolution} and Theorem 7.37 in \cite{l96} yields the pointwise bound
\[
\frac{2\left(\alpha-\Delta_{\hat{e}}^2u^{\lambda}\right)}{\alpha-\beta}\ge C^{-1}\quad \text{in} \quad K_+, 
\]
where $C>0$ is a constant. Consequently
\[
\Delta_{\hat{e}}^2u^{\lambda}\le \alpha-\frac{\alpha-\beta}{2C} \quad \text{in}\quad K_+.
\]
This leads to the contradiction
\[
\alpha=\sup\limits_{\mathbb{R}^{n+1}_-}\Delta_e^2u=\lim_{\lambda\rightarrow\infty}\sup\limits_{K_+}\Delta_{\hat{e}}^2u^{\lambda}\le \alpha-\frac{\alpha-\beta}{2C}<\alpha.
\]
Thus $\alpha=\beta$. The proposition has been proved.
\end{proof}

We give the exact value of $\Delta_k^tu$.

\begin{prop}\label{deltat}
For every $k\in\mathbb{N}^+$, we have the identity
\[
\Delta_k^tu\equiv-\tau.
\]
\end{prop}

\begin{proof}
Denote $k_\lambda:=\frac{k}{\lambda^2}$. We claim that 
\begin{equation}\label{tclaim}
\lim_{\lambda\rightarrow\infty}\int_{K_-}\Delta_{k_\lambda}^tu^\lambda dxdt=\int_{K_-}(-\tau)dxdt=-\tau\left|K_-\right|.
\end{equation}
Actually
\[
\begin{aligned}
\int_{K_-}\Delta_{k_\lambda}^tu^\lambda dxdt
&=\int_{B_1}\int_{-5}^{-4}\frac{u^\lambda(x,t)-u^\lambda(x,t-k_\lambda)}{k_\lambda}dxdt\\
&=k_\lambda^{-1}\int_{B_1}\left(\int_{-5}^{-4}u^\lambda(x,t)dt-\int_{-5-k_\lambda}^{-4-k_\lambda}u^\lambda(x,t)dt\right)dx\\
&=k_\lambda^{-1}\int_{B_1}\left( \int^{-4}_{-4-k_\lambda}u^\lambda(x,t)dt-\int^{-5}_{-5-k_\lambda}u^\lambda(x,t)dt \right)dx.
\end{aligned}
\]	
We also have that
\[
\begin{aligned}
\int_{K_-}(-\tau)dxdt=k_\lambda^{-1}\int_{B_1}\left( \int^{-4}_{-4-k_\lambda}Q(x,t)dt-\int^{-5}_{-5-k_\lambda}Q(x,t)dt \right)dx.
\end{aligned}
\]
It follows from Lemma \ref{gradient holder} that
\[
\begin{aligned}
&\left|\int_{K_-}\Delta_{k_\lambda}^tu^\lambda dxdt-\int_{K_-}(-\tau)dxdt\right|\\
&\le k_\lambda^{-1}\int_{B_1}\left( \int^{-4}_{-4-k_\lambda}\left|u^\lambda-Q\right|(x,t)dt+\int^{-5}_{-5-k_\lambda}\left|u^\lambda-Q\right|(x,t)dt \right)dx\rightarrow0
\end{aligned}
\]
as $\lambda\rightarrow\infty$. The claim has been proved. 

To prove the proposition, we just need to get that
\begin{equation}\label{twoinequalities}
\sup_{\mathbb{R}^{n+1}_-}\Delta_k^tu=-\tau=\inf_{\mathbb{R}^{n+1}_-}\Delta_k^tu. 
\end{equation}
We only give the proof of the first equality in \eqref{twoinequalities} and the second is similar. Denote 
\[
v:=u+m_2t,\quad v^\lambda(x,t):=\lambda^{-2}v\left(\lambda x,\lambda^2 t\right),\quad \gamma:=\sup\limits_{\mathbb{R}^{n+1}_-}\Delta_k^tv.
\]
From \eqref{ut} we know that 
\[
-m_2\le\inf_{\mathbb{R}^{n+1}_-}\Delta_k^tu\le -\tau \le\sup_{\mathbb{R}^{n+1}_-}\Delta_k^tu\le -m_1\quad \text{in}\quad\mathbb{R}^{n+1}_-,
\]
where the last inequality gives that $\gamma<\infty$ and the first one implies that
\begin{equation}\label{ge01}
\Delta_{k_\lambda}^tv^\lambda=\Delta_{k_\lambda}^tu^\lambda+m_2\ge0.
\end{equation}
\eqref{tclaim} implies that
\[
\lim_{\lambda\rightarrow\infty}\int_{K_-}\Delta_{k_\lambda}^tv^\lambda dxdt=\left(m_2-\tau\right)\left|K_-\right|.
\]
Hence we know that $\gamma\ge m_2-\tau$. Suppose $\gamma>m_2-\tau$, then by \eqref{ge01}
\[
\begin{aligned}
&\limsup_{\lambda\rightarrow\infty}\left( \frac{\gamma+m_2-\tau}{2}\left|\left\{ \Delta_{k_\lambda}^tv^{\lambda}\ge \frac{\gamma+m_2-\tau}{2} \right\}\cap K_- \right| \right)\\
&\le \lim_{\lambda\rightarrow\infty}\int_{K_-}\Delta_{k_\lambda}^tv^{\lambda} dxdt=\left(m_2-\tau\right)\left|K_-\right|.
\end{aligned}
\]
Thus for $\lambda$ large
\[
\frac{\left|\left\{ \Delta_{k_\lambda}^tv^{\lambda}\le \frac{\gamma+m_2-\tau}{2} \right\}\cap K_- \right|}{\left|K_-\right|}\le\frac{2\left(m_2-\tau\right)}{\gamma+m_2-\tau},
\]
which means that
\[
\frac{\left|\left\{\frac{2\left(\gamma-\Delta_{k_\lambda}^tv^{\lambda}\right)}{\gamma-m_2+\tau}\ge1\right\}\cap K_-\right|}{\left|K_-\right|}\ge \frac{\gamma-m_2+\tau}{\gamma+m_2-\tau}.
\] 
By Theorem 7.37 in \cite{l96} and \eqref{tsubsolution} we arrive at 
\[
\frac{2\left(\gamma-\Delta_{k_\lambda}^tv^{\lambda}\right)}{\gamma-m_2+\tau}\ge C^{-1}\quad \text{in} \quad K_+
\]
for a positive constant $C>0$. In other words
\[
\Delta_{k_\lambda}^tv^{\lambda}\le \gamma-\frac{\gamma-m_2+\tau}{2C} \quad \text{in}\quad K_+.
\]
It follows that
\[
\gamma=\sup\limits_{\mathbb{R}^{n+1}_-}\Delta_{k}^tv=\lim_{\lambda\rightarrow\infty}\sup\limits_{K_+}\Delta_{k_\lambda}^tv^{\lambda}\le \gamma-\frac{\gamma-m_2+\tau}{2C}<\gamma,
\]
which is a contradiction. Therefore $\gamma=m_2-\tau$.
\end{proof}

Choose $b\in\mathbb{R}^n$ so that the function
\[
w(x,t):=u(x,t)-Q(x,t)-b\cdot x
\]
satisfies $w(e_k,0)=w(-e_k,0)$ for $1\le k\le n$. Clearly $w\left(0_n,0\right)=0$. Due to Proposition \ref{deltax}, there is
\[
\Delta_{e_k}^2w\le0
\] 
for $1\le k\le n$. This inequality, combined with Lemma A.3 in \cite{cl04}, yields the fundamental estimate
\begin{equation}\label{w<=0}
w(je_k,0)\le0\quad \text{for}\quad 1\le k\le n \quad \text{and}\quad j\in\mathbb{Z}.
\end{equation}
By Proposition \ref{deltat}, we arrive at
\[
\Delta_1^tw=0 \quad \text{in} \quad \mathbb{R}^{n+1}_-,
\] 
which means that $w$ has 1-periodicity with respect to $t$ variable. There is
\[
-(w+Q)_t\det D_x^2(w+Q)=f\quad\text{in}\quad \mathbb{R}^{n+1}_-,
\]
together with uniform bounds
\begin{equation}\label{twouniformbounds1}
m_1\le-(w+Q)_t\le m_2,\quad C^{-1}I\le D_x^2(w+Q)\le CI\quad\text{in}\quad \mathbb{R}^{n+1}_-.
\end{equation}

Now we construct a periodic function satisfying the same equation with $w$. Applying Theorem \ref{existence}, there exists $\tilde{w}\in C^{4+\alpha,\frac{4+\alpha}{2}}\left(\mathbb{R}^{n+1}_-\right)$ solving 
\[
-\left(\tilde{w}+Q\right)_t\det D_x^2\left(\tilde{w}+Q\right)=f\quad\text{in}\quad\mathbb{R}^{n+1}_-.
\]
Additionally there exists a constant $C>1$ such that
\begin{equation}\label{twouniformbounds2}
C^{-1}\le-\left(\tilde{w}+Q\right)_t\le C,\quad C^{-1}I\le D_x^2\left(\tilde{w}+Q\right)\le CI\quad\text{in}\quad \mathbb{R}^{n+1}_-.
\end{equation}

Define the difference $h:=w-v$. According to \eqref{w<=0} and the periodicity of $\tilde{w}$
\begin{equation}\label{h<=0}
h\left(je_k,0\right)\le0\quad\text{for}\quad 1\le k\le n \quad \text{and}\quad j\in\mathbb{Z}.
\end{equation}
$h$ satisfies linear parabolic equation 
\[
h_t-a_{ij}D_{x_ix_j}h=0\quad \text{in}\quad \mathbb{R}^{n+1}_-, 
\]
where
\[
a_{ij}(x,t):=\frac{\int_0^1\left[sD_x^2(w+Q)+(1-s)D_x^2(v+Q)\right]^{ij}ds}{\int_0^1\left[-s(w+Q)_t+(s-1)(v+Q)_t\right]^{-1}ds}.
\]
By \eqref{twouniformbounds1} and \eqref{twouniformbounds2}, the equation is uniformly parabolic: 
\[
C^{-1}I\le \left(a_{ij}(x,t)\right)\le CI \quad \text{in}\quad \mathbb{R}^{n+1}_-.
\]
The temporal derivative is also bounded:
\begin{equation}\label{h_t}
-C\le h_t\le C,
\end{equation}
where $C>1$ is a constant. From the periodicity of $v$ and $w$, $h$ is periodic in $t$. 

We prove that $h$ has a finite supremum.
\begin{prop}\label{h}
$\sup\limits_{\mathbb{R}^{n+1}_-}h<\infty$.
\end{prop}

\begin{proof}
Let $M_i:=\sup\limits_{\left[-i,i\right]^n\times\left[-i^2,0\right]}h$. Suppose the contrary. Then $\lim\limits_{i\rightarrow\infty}M_i=\infty$. By the proof of Lemma 4.5 in \cite{ll22}, we have that
\begin{equation}\label{M_i}
\sup_{\left[-2^i,2^i\right]^n}h(x,0)\le 4\sup_{\left[-2^{i-1},2^{i-1}\right]^n}h(x,0)+C_1,\quad i\ge3
\end{equation}
for $C_1>0$ independent of $i$. By \eqref{h_t} and the periodicity of $h$ in $t$, we obtain that
\begin{equation}\label{iterate}
M_{2^i}\le\sup_{\left[-2^i,2^i\right]^n}h(x,0)+C_0\le 4\sup_{\left[-2^{i-1},2^{i-1}\right]^n}h(x,0) +C_0+C_1\le 4M_{2^{i-1}}+C.
\end{equation}

Define the rescaled functions
\[
H_i(x,t):=\frac{h\left(2^ix,2^{2i}t\right)}{M_{2^i}}\quad\text{for}\quad(x,t)\in\left[-1,1\right]^n\times\left(-\infty,0\right].
\]
It is obvious that $H_i\left(0_n,0\right)=0$ and $H_i\le 1$ in $\left[-1,1\right]^n\times\left(-\infty,0\right]$. There is also
\[
H_{i}\left(\pm\frac{1}{2}e_k,0\right)=\frac{h\left(\pm2^{i-1}e_k,0\right)}{M_{2^i}}\le0,\quad 1\le k\le n
\]
from \eqref{h<=0}. From the iteration inequality \eqref{iterate}, we have that for large $i$
\[
\sup_{\left[-\frac{1}{2},\frac{1}{2}\right]^n\times\left[-\frac{1}{4},0\right]}H_i=\frac{M_{2^{i-1}}}{M_{2^i}}\ge\frac{M_{2^i}-6C}{4M_{2^i}}\ge\frac{1}{8}.
\]
Applying the Harnack inequality in \cite{l96} (Corollary 7.42) and periodicity of $H_i$
\[
\sup_{\left[-\frac{8}{9},\frac{8}{9}\right]^n\times\left[-1,0\right]}\left(1-H_i\right)\le C\left(1-H_i\left(0_n,0\right)\right)=C,
\]
which yields that $1-C\le H_i\le 1$ in $\left[-\frac{8}{9},\frac{8}{9}\right]^n\times\left[-1,0\right]$. By the Hölder estimate in \cite{l96} (Theorem 7.41), there exists $\beta\in(0,1)$, such that
\[
\left\|H_i\right\|_{C^{\beta,\frac{\beta}{2}}\left(\left[-\frac{3}{4},\frac{3}{4}\right]^n\times\left[-\frac{3}{4},0\right]\right)}\le C.
\]
Take $\alpha\in(0,\beta)$, and it follows from Arzela-Ascoli theorem that there exists $H\in C^{\alpha,\frac{\alpha}{2}}\left(\left[-\frac{3}{4},\frac{3}{4}\right]^n\times\left[-\frac{3}{4},0\right]\right)$, such that
\[
H_i\rightarrow H\quad\text{in}\quad C^{\alpha,\frac{\alpha}{2}}\left(\left[-\frac{3}{4},\frac{3}{4}\right]^n\times\left[-\frac{3}{4},0\right]\right)
\]
along a subsequence as $i\rightarrow\infty$. 

We claim that $H=H(x)$ and $H$ is concave. Actually, we notice that 
\begin{equation}\label{almostconst}
\Delta_{2^{-2i}k}^tH_i=0
\end{equation}
inherited from $\Delta_k^th=0$. Fix $t_1,t_2\in\left[-\frac{3}{4},0\right]$. For any $i\ge3$ , there exists $k_i\in\mathbb{N}$, such that
\[
\left|t_2-\left(t_1-2^{-2i}k_i\right)\right|\le 2^{-2i}.
\]
It follows from \eqref{almostconst} that
\[
H_i\left(x,t_1\right)=H_i\left(x,t_1-2^{-2i}k_i\right) \quad\text{for}\quad x\in \left[-\frac{3}{4},\frac{3}{4}\right]^n.
\]
Let $i\rightarrow\infty$, and we have $H\left(x,t_1\right)=H\left(x,t_2\right)$. The arbitrariness of $t_1,t_2$ gives the first claim. For the second one, we have that 
\begin{equation}\label{almostconcavity}
\Delta_{2^{-i}e}^2H_i\le0
\end{equation}
induced from $\Delta_e^2h\le0$. Fix $x,y$ such that $x,x-y,x+y\in \left[-\frac{3}{4},\frac{3}{4}\right]^n$. For any $i\ge3$, there exists $e^i\in E$ with $|e^i|\le 2^i\sqrt{n}$, such that $\left|y-2^{-i}e^i\right|\le 2^{-i}\sqrt{n}$, $x+2^{-i}e^i,x-2^{-i}e^i\in \left[-\frac{3}{4},\frac{3}{4}\right]^n$. \eqref{almostconcavity} implies that
\[
H_i\left(x+2^{-i}e^i\right)+H_i\left(x-2^{-i}e^i\right)\le 2H_i\left(x\right).
\]
Let $i\rightarrow\infty$, we have that
\[
H(x+y)+H(x-y)\le 2H(x),
\]
which implies the concavity of $H$.

From the properties of $H_i$, the function $H$ satisfies that
\[
H\left(0_n\right)=0,\quad \sup_{\left[-\frac{1}{2},\frac{1}{2}\right]^n}H\ge\frac{1}{8},\quad H\left(\pm\frac{1}{2}e_k\right)\le0\quad \text{for} \quad 1\le k\le n.
\]
Since $H$ is concave and $H\left(0_n\right)=0$, there exists a linear function $l$, such that $l\ge H$ in $\left[-\frac{3}{4},\frac{3}{4}\right]^n$ and $l\left(0_n\right)=0$. For any $\epsilon>0$, there exists $i'$, such that
\[
\left\| H-H_{i'} \right\|_{C^{\alpha,\frac{\alpha}{2}}\left(\left[-\frac{3}{4},\frac{3}{4}\right]^n\times\left[-\frac{3}{4},0\right]\right)}< \epsilon.
\]
Then by Harnack inequality in \cite{l96} (Corollary 7.42) 
\[
\sup\limits_{\left[-\frac{3}{4},\frac{3}{4}\right]^n\times\left[-\frac{3}{4},0\right]}(l-H)\equiv0,
\]
which means that $H\equiv l$. Because $H\left(0_n\right)=0$ and $H\left(\pm\frac{1}{2}e_k,0\right)\le0$ for $1\le k\le n$, $H\equiv l\equiv0$ in $\left[-\frac{3}{4},\frac{3}{4}\right]^n$. This contradicts $\sup\limits_{\left[-\frac{1}{2},\frac{1}{2}\right]^n}H\ge\frac{1}{8}$! Thus $\sup\limits_{\mathbb{R}^{n+1}_-}h<\infty$.
\end{proof}

\begin{proof}[proof of Theorem \ref{liouville}]
The periodicity in the temporal variable t and the boundedness of $h$ implies that $h$ must be constant via the Harnack inequality.
\end{proof}

\noindent {\bf Funding:} J. Bao is supported by the National Natural Science Foundation of China (12371200).

\noindent {\bf Conflict of Interest:} The authors declare that they have no conflict of interest.

\medskip

\bibliographystyle{plain}
\def\cprime{$'$}

\end{document}